\documentclass[leqno,11pt]{article}

\usepackage[utf8]{inputenc}
\usepackage[T1]{fontenc}
\usepackage[english]{babel}

\usepackage[intlimits]{amsmath}
\usepackage{amssymb, amsfonts}
\usepackage{amsthm}

\usepackage{MnSymbol}
\usepackage{mathbbol}
\usepackage{bbm}
\usepackage{mathrsfs}

\usepackage[all]{xy}

\usepackage{graphicx}
\usepackage{color}
\usepackage{wrapfig}

\numberwithin{equation}{section}

\newtheorem{theoremcounter}{theoremcounter}[section]
\newtheorem{thmstarcounter}{thmstarcounter}

\newtheorem*{conjecturestar}{Conjecture}
\newtheorem{corollary}[theoremcounter]{Corollary}
\newtheorem{lemma}[theoremcounter]{Lemma}

\newtheorem{proposition}[theoremcounter]{Proposition}
\newtheorem{theorem}[theoremcounter]{Theorem}
\newtheorem{thmstar}[thmstarcounter]{Theorem}

\theoremstyle{definition}

\newtheorem{definition}[theoremcounter]{Definition}

\newtheorem{example}[theoremcounter]{Example}
\newtheorem{notation}[theoremcounter]{Notation}

\newtheorem*{questionstar}{Question}
\newtheorem{remark}[theoremcounter]{Remark}

\newcommand{\cA}{\ensuremath{\mathcal{A}}}

\newcommand{\cN}{\ensuremath{\mathcal{N}}}
\newcommand{\cO}{\ensuremath{\mathcal{O}}}

\newcommand{\cU}{\ensuremath{\mathcal{U}}}

\newcommand{\rD}{\ensuremath{\mathrm{D}}}
\newcommand{\rE}{\ensuremath{\mathrm{E}}}

\newcommand{\rG}{\ensuremath{\mathrm{G}}}

\newcommand{\rK}{\ensuremath{\mathrm{K}}}
\newcommand{\rL}{\ensuremath{\mathrm{L}}}
\newcommand{\rM}{\ensuremath{\mathrm{M}}}

\newcommand{\rR}{\ensuremath{\mathrm{R}}}
\newcommand{\rS}{\ensuremath{\mathrm{S}}}

\newcommand{\rV}{\ensuremath{\mathrm{V}}}

\newcommand{\rmd}{\ensuremath{\mathrm{d}}}

\newcommand{\rmm}{\ensuremath{\mathrm{m}}}

\newcommand{\veps}{\ensuremath{\varepsilon}}

\newcommand{\vphi}{\ensuremath{\varphi}}

\newcommand{\ol}{\overline}

\newcommand{\amid}{\ensuremath{\, | \,}}

\newcommand{\eqstop}{\ensuremath{\, \text{.}}}
\newcommand{\eqcomma}{\ensuremath{\, \text{,}}}

\newcommand{\NN}{\ensuremath{\mathbb{N}}}
\newcommand{\ZZ}{\ensuremath{\mathbb{Z}}}
\newcommand{\QQ}{\ensuremath{\mathbb{Q}}}
\newcommand{\RR}{\ensuremath{\mathbb{R}}}
\newcommand{\CC}{\ensuremath{\mathbb{C}}}

\newcommand{\ra}{\ensuremath{\rightarrow}}

\newcommand{\Aut}{\ensuremath{\mathrm{Aut}}}

\newcommand{\sslash}{\mathbin{/\mkern-6mu/}}

\newcommand{\Cstar}{\ensuremath{\text{C}^*}}
\newcommand{\Wstar}{\ensuremath{\text{W}^*}}

\newcommand{\bo}{\ensuremath{\mathcal{B}}}

\newcommand{\supp}{\ensuremath{\mathop{\mathrm{supp}}}}

\newcommand{\Cstarred}{\ensuremath{\Cstar_\mathrm{red}}}
\newcommand{\Cstarmax}{\ensuremath{\Cstar_\mathrm{max}}}

\newcommand{\contc}{\ensuremath{\mathrm{C}_\mathrm{c}}}

\newcommand{\Ltwo}{\ensuremath{{\offinterlineskip \mathrm{L} \hskip -0.3ex ^2}}}

\newcommand{\ltwo}{\ensuremath{\ell^2}}

\newcommand{\Ad}{\ensuremath{\mathop{\mathrm{Ad}}}}
\newcommand{\Out}{\ensuremath{\mathrm{Out}}}

\newcommand{\grpaction}[1]{\ensuremath{\stackrel{#1}{\curvearrowright}}}

\newcommand{\freegrp}[1]{\ensuremath{\mathbb{F}}_{#1}}

\begin{document}

\begin{center}
  \textbf{\LARGE \Cstar-simplicity of locally compact Powers groups}  \\[1em]
  Sven Raum
  \footnote{The research leading to these results has received funding from the People Programme (Marie Curie Actions) of the European Union's Seventh Framework Programme (FP7/2007-2013) under REA grant agreement n°[622322].}

\end{center}


\begin{abstract}
  In this article we initiate research on locally compact \Cstar-simple groups.  We first show that every \Cstar-simple group must be totally disconnected.  Then we study \Cstar-algebras and von Neumann algebras associated with certain groups acting on trees.  After formulating a locally compact analogue of Powers' property, we prove that the reduced group \Cstar-algebra of such groups is simple.  This is the first simplicity result for \Cstar-algebras of non-discrete groups and answers a question of de la Harpe.  We also consider group von Neumann algebras of certain non-discrete groups acting on trees.  We prove factoriality, determine their type and show non-amenability.  We end the article by giving natural examples of groups satisfying the hypotheses of our work.
\end{abstract}

\renewcommand{\thefootnote}{}
\footnotetext{}
\footnotetext{\textit{MSC classification:} Primary 22D25 ; Secondary 46L05, 46L10, 20C07, 20C08 }
\footnotetext{\textit{Keywords:} \Cstar-simplicity, group factor, non-amenability, totally disconnected group, Schlichting completion \\}

\section{Introduction}
\label{sec:introduction}

Group \Cstar-algebras and group von Neumann algebras enjoy a long history.  Group von Neumann algebras of discrete groups have been used by McDuff \cite{mcduff69-uncountably} to provide examples of a continuum of pairwise non-isomorphic ${\rm II}_1$-factors.  Connes' conjecture about \Wstar-superrigidity of group von Neumann algebras of discrete groups with property (T) \cite{connes82} drew analogues with lattices in their ambient Lie groups.  Only very recently in a breakthrough result Ioana-Popa-Vaes \cite{ioanapopavaes10} could provide the first examples of so called \Wstar-superrigid groups at all, leaving Connes' conjecture untouched.  Group von Neumann algebras equally well found applications in topology, where \mbox{$\Ltwo$-Betti} numbers can be defined thanks to a continuous notion of dimension provided by the operator algebraic viewpoint \cite{lueck01-book}.  However, until now group von Neumann algebras associated with non-discrete groups drew only minor attention.

Group \Cstar-algebras of discrete groups have been investigated with a focus on the Baum-Connes conjecture and on simplicity results, after Kadison asked whether the reduced group \Cstar-algebra $\Cstarred(\freegrp{2})$ is simple without non-trivial projections \cite{powers75}.  Powers in \cite{powers75} showed with combinatorial methods that $\Cstarred(\freegrp{2})$ is simple.  He used an averaging argument that formed the basis of most follow up results on \Cstar-simplicity.  His argument was put in an abstract context by de la Harpe and lead for example to simplicity results for free products, hyperbolic groups and Baumslag-Solitar groups \cite{delaharpe07-survey,delaharpepreaux11}.  Very recently, the astonishing work of Kalantar-Kennedy and Breuillard-Kalantar-Kennedy-Ozawa \cite{breuillardkalantarkennedyozawa14} basically pushed the question of \Cstar-simplicity to the point, where the only serious open problem had to be solved by group theorists.  Shortly afterwards, the conjecture that a discrete group is \Cstar-simple if and only if it has a trivial amenable radical could be proven wrong by Le Boudec in \cite{leboudec15-cstarsimplicity}.  Further, in the same year Kennedy and Haagerup gave a characterisation of \Cstar-simple groups.  While Kennedy provides a group theoretic characterisation in terms of recurrent amenable subgroups \cite{kennedy15-cstarsimplicity}, Haagerup proves the equivalence of \Cstar-simplicity and the Powers averaging property in \cite{haagerup15}.

Also in representation theory group \Cstar-algebras had a serious impact (see e.g. \cite{rosenberg94}).  The probably easiest approach to classical Peter-Weyl theory studies the group \mbox{\Cstar-algebra} $\Cstar(K)$ of a compact group $K$ .  Further the notion of a type ${\rm I}$ group stems from \Cstar-algebras \cite{kaplansky51,glimm61}.  Their representation theory is completely understood by its irreducible representations \cite{dixmier77}.  Examples of type ${\rm I}$ groups include connected semisimple Lie groups \cite{harish-chandra66}, reductive algebraic groups over non-Archimedean fields \cite{bernstein74-type-I} and the full group of automorphisms of a tree \cite{figa-talamancanebbia91}.  Also \Cstar-simplicity has a representation theoretic interpretation, invoking the Fell topology on unitary representations \cite{fell60}.  In fact, a group is \Cstar-simple if and only if the support of its left regular representation is a closed point in the group's unitary dual.

In this article, we initiate the study of non-discrete \Cstar-simple groups.  Our first result shows that \mbox{every} \Cstar-simple group is totally disconnected, extending a result of Bekka-Cowling-de~la~Harpe. \cite[Proposition~4]{bekkacowlingdelaharpe94}
\begin{thmstar}[Theorem \ref{thm:cstar-simple-implies-td}]
  Let $G$ be a locally compact \Cstar-simple group.  Then $G$ is totally disconnected.
\end{thmstar}

We then provide first examples of non-discrete \Cstar-simplicity groups, answering a question of de la Harpe \cite[Question 5]{delaharpe07-survey}.

\begin{questionstar}[de la Harpe]
  Does there exist a non-discrete second countable locally compact group which is \Cstar-simple?
\end{questionstar}

We adapt the combinatorial method of Powers averaging from a discrete setting to the setting of general totally disconnected groups.  Then, inspired by work of de la Harpe-Pr{\'e}aux, we study the action of a closed subgroup of $\Aut(T)$ on the boundary $\partial T$ of the tree, in order to obtain by geometric means the necessary input to apply Powers averaging to group \Cstar-algebras of some natural class of non-discrete locally compact groups.  The main achievement of this article is hence twofold.  On the one hand, we are able to answer de la Harpe's question, giving examples of \Cstar-simple non-discrete second countable locally compact groups.  On the other hand, we show how averaging techniques, well known for discrete groups, can be exploited in operator algebras associated with locally compact totally disconnected groups.

The class of groups treated in this article, has a geometric and an algebraic formulation.  We introduce the following notation, where $\cN_G(K)$ denotes the normaliser of a subgroup $K \leq G$ and $G_0$ denotes the kernel of the modular function of a locally compact group.  Recall that a tree is called thick if each of its vertices has valency at least 3.

\textbf{Notation} (See Theorem \ref{thm:characterisation-condition-star})\textbf{.}
  We say that a locally compact group $G$ satisfies condition ($*$) if one of the following equivalent conditions holds.
  \begin{itemize}
  \item There is a locally finite tree $T$ such that $G \leq \Aut(T)$ as a closed subgroup.  Further, $G$ is non-amenable without any non-trivial compact normal subgroup and there is a compact open subgroup $K \leq G$ such that $\cN_G(K)/K$ contains an element of infinite order.
  \item There is a locally finite thick tree $T$ such that $G \leq \Aut(T)$ as a closed subgroup.  Further, $G$ acts minimally on $\partial T$ and there is some point $x \in \partial T$ such that $G_x \leq G$ is open and $G_x \cap G_0$ contains a hyperbolic element.
  \end{itemize}

We can now formulate the main theorem of this article.
\begin{thmstar}[Theorem \ref{thm:cstar-simplicity}]
  \label{thm:introduction:cstar-simplicity}
  Let $G$ be a group satisfying condition ($*$).  Then $\Cstarred(G)$ is simple.  
\end{thmstar}

We give two applications of our main result.  Since group \Cstar-algebras of totally disconnected groups and Hecke \Cstar-algebras are related, our result also implies simplicity of certain reduced Hecke \Cstar-algebras (Section \ref{sec:hecke-cstar-algebras}), which is in sharp contrast to results previously obtained.
\begin{thmstar}[Corollary \ref{cor:simple-Hecke-algebra}]
  Let $T$ be thick tree and $\Gamma \leq \Aut(T)$ some not necessarily closed group acting without proper invariant subtree. Let $\Lambda$ be some vertex stabiliser in $\Gamma$ and assume that there is a finite index subgroup $\Lambda_0 \leq \Lambda$ such that $\cN_\Gamma(\Lambda_0)/\Lambda_0$ contains an element of infinite order.  Then the reduced Hecke \Cstar-algebra $\Cstarred(\Gamma, \Lambda)$ is simple.
\end{thmstar}

As a further corollary of Theorem \ref{thm:introduction:cstar-simplicity}, we obtain a result about type ${\rm I}$ groups acting on trees.   A locally compact group $G$ is called type ${\rm I}$ group if each of its unitary representations generates a type ${\rm I}$ von Neumann algebra.  

\begin{thmstar}[Corollary \ref{cor:not-type-I}]
  Let $T$ be a thick tree and $G \leq \Aut(T)$ be a closed subgroup acting minimally on $\partial T$.  Assume that there is $x \in \partial T$ such that 
  \begin{itemize}
  \item $Kx$ is finite for some compact open subgroup $K \leq G$, and
  \item there is some hyperbolic element in $G_0 \cap G_x$.
  \end{itemize}
  Then $G$ is not a type ${\rm I}$ group.
\end{thmstar}

Note that our result is in contrast to and motivated by the following conjecture on type ${\rm I}$ groups. 

\begin{conjecturestar}
  Let $G \leq \Aut(T)$ be a closed subgroup.  Assume that there is a compact open subgroup of $G$ acting transitively on $\partial T$.  Then $G$ is a type ${\rm I}$ group.
\end{conjecturestar}

Since to a certain extend we are able to control weights on a group \Cstar-algebra, our methods are also able to deal with von Neumann algebraic results.  We obtain the following factoriality results for group von Neumann algebras of non-discrete groups and we are able to determine their type in terms of the modular function.  We refer the reader to Section \ref{sec:type} for some facts about Connes' $\rS$-invariant $\rS(M)$ for a factor $M$.
\begin{thmstar}[Theorem \ref{thm:factoriality}]
  Let $G$ be a group satisfying condition ($*$).  Further assume that some compact open subgroup of $G$ is topologically finitely generated.  Then the group von Neumann algebra $\rL(G)$ is a factor and $\rS(\rL(G)) = \ol{\Delta(G)}$.
  \begin{itemize}
  \item If $G$ is discrete, then $\rL(G)$ is a type ${\rm II}_1$ factor.
  \item If $G$ is unimodular but not discrete, then $\rL(G)$ is a type ${\rm II}_\infty$ factor.
  \item If $\Delta(G) = \lambda^\ZZ$ for some $\lambda \in (0,1)$, then $\rL(G)$ is a type ${\rm III}_\lambda$ factor.
  \item If $\Delta(G)$ is not singly generated, then $\rL(G)$ is a type ${\rm III}_1$ factor.
  \end{itemize}
\end{thmstar}

We also prove non-amenability of the group von Neumann algebra associated with certain groups acting on trees.
\begin{thmstar}[Theorem \ref{thm:non-amenability}]
  Let $G$ be a group satisfying condition (*).  Further assume that some compact open subgroup of $G$ is topologically finitely generated.  Then $\rL(G)$ is not amenable.
\end{thmstar}

Finally, we give concrete examples of groups satisfying the hypotheses of our work.  They arise as Schlichting completions of Baumslag-Solitar groups.
\begin{thmstar}[Theorem \ref{thm:BS-groups}]
  Let $2 \leq |m| \leq n$ and consider the relative profinite completion $\rG(m,n)$ of the Baumslag-Solitar group $\mathrm{BS}(m,n)$.  Then the following statements are true.
  \begin{itemize}
  \item $\rL(\rG(m,n))$ is a non-amenable factor.
  \item If $|m| = n$, then $\rG(m.n)$ is discrete and $\rL(G(m,n))$ is of type ${\rm II}_1$.
  \item If $|m| \neq n$, then $\rL(G(m,n))$ is of type ${\rm III}_{\left | \frac{m}{n} \right |}$.
  \item $\Cstarred(\rG(m,n))$ is simple.
  \end{itemize}
 
\end{thmstar}
The fact that $\rL(\rG(m,n))$ is a factor and the calculation of its type was obtained with different methods in unpublished work of the author and C.Ciobotaru.

\section*{Acknowledgements}

We want to thank Alain Valette for his hospitality at the University of Neuch{\^a}tel, where part of this work was done.  We are grateful to Pierre-Emmanuel Caprace for useful comments on groups acting on trees. We thank Siegfried Echterhoff for asking us whether \Cstar-simple groups are totally disconnected and for a helpful discussion about this question.  Further we thank Hiroshi Ando, Pierre de la Harpe, Pierre Julg and Stefaan Vaes for useful comments on the first version this article.  Finally we thank the anonymous referee his comments and for suggesting an easier proof of Theorem \ref{thm:cstar-simple-implies-td}.

\section{Preliminaries}
\label{sec:preliminaries}

\subsection{Totally disconnected groups}
\label{sec:totoally-disconnected-groups}

For a locally compact group $G$, we denote by $G_0$ the kernel of its \emph{modular function} ${\Delta:G \ra \RR_{> 0}}$ determined by $\mu(Ag) = \Delta(g) \mu(A)$ for any left Haar measure $\mu$ on $G$ and any measurable set $A \subset G$ with finite non-zero Haar measure.  By van Dantzig's theorem \cite[TG 39]{vandantzig36-topologische-algebra-iii}, a locally compact group $G$ is \emph{totally disconnected} if and only if $e \in G$ admits a neighbourhood basis of compact open subgroups.  The modular function $\Delta$ of a totally disconnected locally compact group $G$ with left Haar measure $\mu$ satisfies
\begin{equation*}
  \Delta(g)
  =
  \frac{\mu(g^{-1}(K \cap g Kg^{-1})g)}{\mu(K \cap g Kg^{-1})}
  =
  \frac{\mu(g^{-1}(K \cap g Kg^{-1})g)}{\mu(K)} \cdot 
  \frac{\mu(K)}{\mu(K \cap g Kg^{-1})}
  =
  \frac{[K: K \cap gK g^{-1}]}{[K : K \cap g^{-1} K g]}
  \in 
  \QQ
  \eqcomma
\end{equation*}
for all compact open subgroups $K \leq G$.  In particular, the modular function of a totally disconnected locally compact group takes only values in $\QQ$.

We need the following two observations, about topologically finitely generated profinite groups.
\begin{lemma}
  \label{lem:control-number-subgroups}
  Let $G$ be a topologically finitely generated group.  Then for every $r \in \NN$ there are only finitely many closed subgroups of index $r$.
\end{lemma}
\begin{proof}
  Let $G$ be a topological group.  If $H \leq G$ is a closed subgroup of index $n < \infty$,  then we can identify $G/H \cong \{1, \dotsc, n\}$.  We obtain a continuous homomorphism $\pi: G \ra \rS_n$ such that $H = \pi^{-1}((\rS_n)_i)$ for some $i \in \{1, \dotsc, n\}$, where $(S_n)_i$ is the stabiliser group of $i$.

  Now assume that $G$ is topologically finitely generated.  Then there are only finitely many continuous homomorphism $G \ra \rS_n$ for each $n$, since the image of $G$ is determined by the image of its generators.  Consequently, $G$ has only finitely many closed subgroups of finite index.
\end{proof}

\begin{proposition}
  \label{prop:characterisation-top-finitely-generated}
  Let $G$ be a locally compact group. If some compact open subgroup of $G$ is topologically finitely generated, then all compact open subgroups of $G$ are topologically finitely generated.
\end{proposition}
\begin{proof}
  Assume that some compact open subgroup $K \leq G$ is topologically.  Any other compact open subgroup of $G$ is commensurated with $K$, that is any compact open subgroup $L \leq G$ satisfies $[K: K \cap L], [L: K \cap L] < \infty$.  So the proposition follows from the fact that topological finite generation passes between finite index inclusions.
\end{proof}

\subsection{Schlichting completions}
\label{sec:schlichting-completions}

Let $\Lambda \leq \Gamma$ be an inclusion of discrete groups.  It is a \emph{discrete Hecke pair} if $[\Lambda : \Lambda \cap g \Lambda g^{-1}] < \infty$ for all $g \in \Gamma$.  We define the map $\iota: \Gamma \ra \mathrm{Sym}(\Gamma/\Lambda)$ by left multiplication $\iota(g) h\Lambda  = gh\Lambda$.  Equipping $\mathrm{Sym}(\Gamma/ \Lambda)$ with the topology of pointwise convergence, we put $\Gamma \sslash \Lambda := \ol{\iota(\Gamma)}$.  This is the \emph{Schlichting completion} of the Hecke pair $\Lambda \leq \Gamma$.  It is a totally disconnected group equipped with the natural compact open subgroup $\ol{\iota(\Lambda)}$.

\subsection{Groups acting on trees}
\label{sec:groups-acting-on-trees}

Let $T$ be a locally finite tree.  Then the group of automorphisms $\Aut(T)$ equipped with the topology of pointwise convergence is a totally disconnected locally compact group.  Every vertex stabiliser is a compact open subgroup of $\Aut(T)$ and every compact subgroup of $\Aut(T)$ stabilises some vertex or some edge of $T$.

The set of one-sided infinite geodesic rays in $T$ modulo equality at some point is called the boundary $\partial T$ of $T$.  Formally, we have
\begin{equation*}
  \partial T
  =
  \{ \xi: \NN \ra T \amid \forall n \in \NN: \, \rmd(\xi(0), \xi(n)) = n\}
  /
  \xi \sim \xi' \text{ if } \exists n_0 \in \NN, m \in \ZZ \, \forall n \geq n_0: \xi(n + m) = \xi'(n)
\eqstop
\end{equation*}
For $\rho, \eta \in T$ we denote by $[\rho, \eta]$, $[\rho, \eta)$, $(\rho, \eta]$ and $(\rho, \eta)$ the set of vertices on the geodesic between $\rho$ and $\eta$, excluding or not their starting and end points.  Similarly, for $x \in \partial T$, we denote by $[\rho, x)$ and $(\rho, x)$ the set of vertices on the unique geodesic which represents $x$ and starts at $\rho$.

The boundary of a tree carries a natural topology called the \emph{shadow topology}.
\begin{definition}
  \label{def:shadow-topology}
  For two vertices $\rho \neq \eta$ of $T$ consider
  \begin{equation*}
    U_{\rho, \eta} := \{x \in \partial T \amid \eta \in [\rho, x)\}
    \eqstop
  \end{equation*}
  Then the topology generated by all $U_{\rho, \eta}$, where $(\rho, \eta)$ runs through all pairs of distinct vertices of $T$, is called the \emph{shadow topology} on $\partial T$.
\end{definition}
The action of $\Aut(T)$ on $T$ induces a continuous action by homeomorphisms on $\partial T$.  We remark that every compact subgroup $K \leq \Aut(T)$ which fixes a point in the boundary, automatically fixes a vertex of $T$.

There are 3 types of elements in $\Aut(T)$.  Elliptic elements, inversions and hyperbolic elements.  An element $g \in \Aut(T)$ is called \emph{elliptic} if it fixes a vertex of $T$.  It is an \emph{inversion}, if it does not fix a vertex of $T$, but an edge.  In all other cases, $g$ is called \emph{hyperbolic}.  Any hyperbolic element fixes exactly two points in $x,y \in \partial T$ and acts by translation along the axis $(x,y)$.  If $g$ translates in the direction of $x$, then $x$ is called the attracting fixed point of $g$.  Two hyperbolic elements are called \emph{transverse} if they have no common fixed point.

\begin{proposition}
  \label{prop:existence-minimal-subtree}
  Let $G \leq \Aut(T)$ be a closed non-amenable subgroup.  Then there is a minimal $G$-invariant subtree $T' \leq T$.  If $G$ does not contain any compact normal subgroup, then $G \grpaction{} T'$ is faithful.
\end{proposition}
\begin{proof}
  Since $G$ is non-amenable, it contains some hyperbolic element.  The smallest tree $T'$ containing all axes of hyperbolic elements in $G$, is $G$-invariant.  Further every $G$-invariant subtree contains $T'$.  Next observe that the kernel of the restriction map $G \mapsto \Aut(T')$ is $G_{T'}$, which is a compact subgroup of $G$.  It follows that $G_{T'}$ is trivial and hence $G \grpaction{} T'$ is faithful if $G$ does not contain any compact normal subgroup.
\end{proof}

\begin{proposition}
  \label{prop:equivalence-minimal-action}
  Let $G \leq \Aut(T)$ be a non-amenable subgroup.  If $G \grpaction{} T$ admits no proper invariant subtree then $G \grpaction{} \partial T$ is minimal.  Vice versa, if $G$ is not compact and $G \grpaction{} \partial T$ is minimal, then $T$ admits no proper invariant subtree.
\end{proposition}
\begin{proof}
  If $G \grpaction{} \partial T$ is not minimal, then there is some $G$-invariant open set $\cU\subset \partial T$ such that $\partial T \setminus \cU$ contains an open subset.  Let
  \begin{equation*}
    T' = \bigcup_{x,y \in \cU} (x,y) \subset T
  \end{equation*}
  be the subtree consisting of all vertices on geodesics joining points in $\cU$.  Then $T' \neq T$, since $\partial T \setminus \cU$ contains an open set.  So we have found a proper $G$-invariant subtree $T' \subset T$

  Now assume that $G$ is not compact and there is a $G$-invariant subtree $T' \leq T$.  Since $G$ is not compact, $T'$ is infinite.  So it contains at least one infinite geodesic ray.   Let $\rho \in T \setminus T'$ be some vertex.  Since $T'$ is convex, there is some neighbouring vertex $\rho \sim \eta \in T \setminus T'$.  So the open set $\cO = U_{\rho, \eta} \subset \partial T$ is not empty.  If $x \in \partial T$ is the endpoint of some geodesic ray in $T'$, then $Gx \cap \cO = \emptyset$.  So the orbit of $x$ is not dense.
\end{proof}

We want to have some control over the action on the boundary $\partial T$ of vertex stabilisers in $\Aut(T)$.  To this end, we make the following definition of the meet of two boundary points with respect to a fixed vertex in $T$.
\begin{notation}
  \label{not:control-shadow-topology}
  Fix a vertex $\rho \in T$.  Then for $x,y \in \partial T$ we set
  \begin{equation*}
    \rmm_\rho(x,y) := \max \{ \rmd(\rho, \eta) \amid \eta \in [\rho, x) \cap [\rho, y)\} \in \NN \cup \{+ \infty\}
    \eqstop
  \end{equation*}
\end{notation}

\begin{remark}
  \label{rem:control-shadow-topology}
  \begin{itemize}
  \item Fix $x \in \partial T$ and $\rho \in T$.  Then a basis of compact open neighbourhoods of $x$ in the shadow topology is given by the sets $\cU_x(n) := \{ y \in \partial T \amid \rmm_\rho(x,y) \geq n\}$, $n \in \NN$.
  \item   The vertex stabiliser $\Aut(T)_\rho$ leaves $\rmm_\rho$ invariant, i.e. for all $x,y \in \partial T$ and all $k \in \Aut(T)_\rho$ we have $\rmm_\rho(kx,ky) = \rmm_\rho(x,y)$.
  \end{itemize}
\end{remark}

We can now quantify the dynamics of a hyperbolic element close to its fixed points.
\begin{proposition}
  \label{prop:contractive}
  Let $g \in \Aut(T)$ be a hyperbolic element fixing a point $x \in \partial T$.  Then for all $\rho \in T$ there is $d \in \NN$ such that for all $y \in \partial T \setminus \{x\}$ with $\rmm_\rho(x,y) \geq d$ we have  
  \begin{equation*}
    \rmm_\rho(x,y) \neq \rmm_\rho(x,gy)
    \eqstop
  \end{equation*}
\end{proposition}

\begin{proof}
  Replacing $g$ by its inverse, we may assume that $x$ is its attracting fixed point.  Denote by $x'$ the repelling fixed point of $g$ and by $\eta$ the vertex of $T$ in which the geodesics $[\rho, x)$ and $[\rho, x')$ split.

  \setlength{\unitlength}{1em}
  \begin{picture}(20,16)
    \put(5,0){\includegraphics[width=15em]{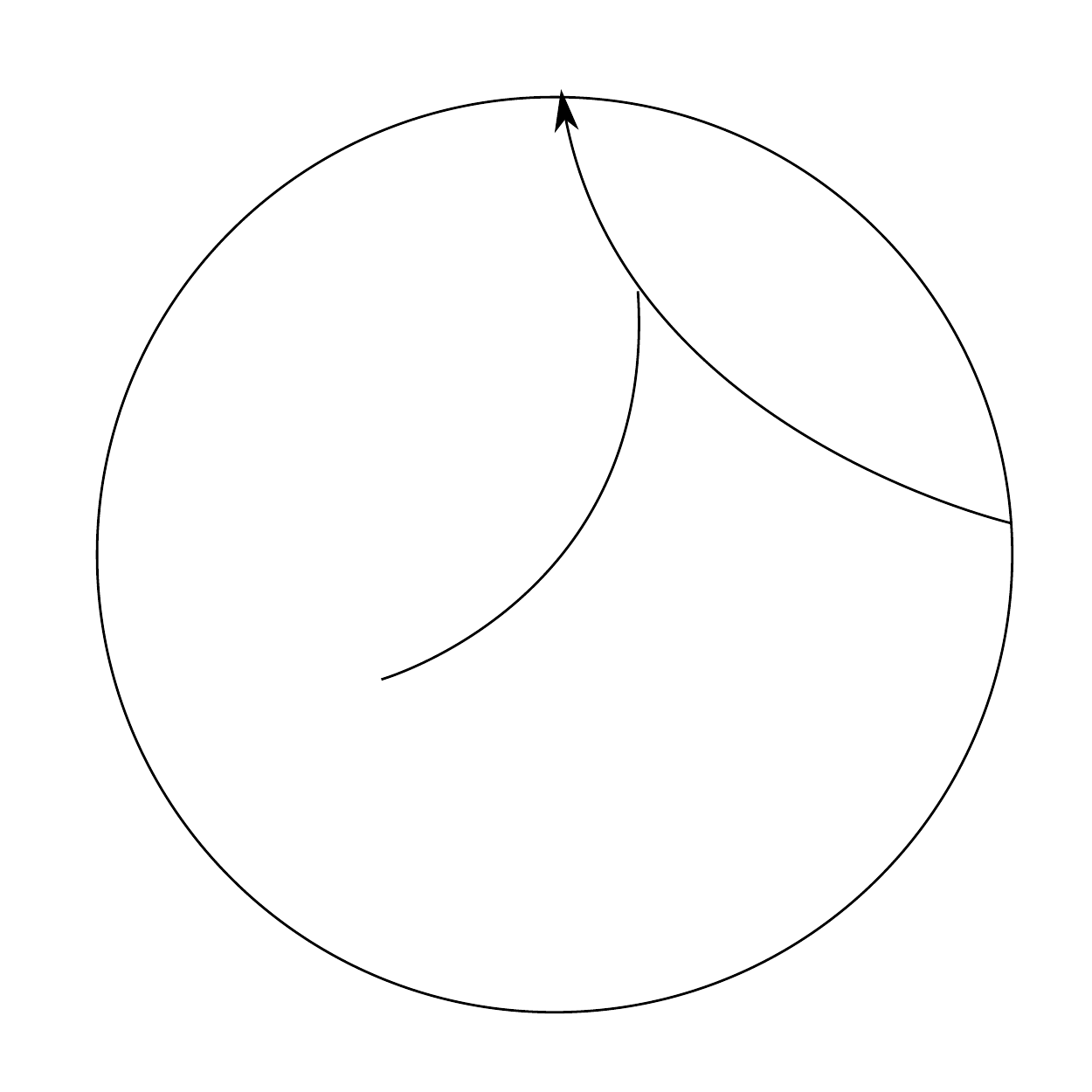}}
    \put(12,14){$x$}
    \put(19.5,8){$x'$}
    \put(14,11){$\eta$}
    \put(9.5,6){$\rho$}
  \end{picture}

  Let $d = \rmd(\rho, \eta) + 1$ and take $y \in \partial T \setminus \{x\}$ such that $\rmm_\rho(x,y) \geq d$. Then the geodesic $[\rho, y)$ passes through $\eta$ in the same direction as $[\rho, x)$, meaning that $[\rho, y) \cap (\eta, x) \neq \emptyset$.  Denote by $\xi$ the vertex in which the geodesic $[\rho, y)$ and $[\rho, x)$ split.  Since $\xi \in (\eta, x) \subset (x, x')$ lies on the axis of $g$, and $x$ is the attracting fixed point of $g$, we have 
  \begin{equation*}
    \rmd(\rho, g \xi)
    =
    \rmd(\rho, \eta) + \rmd(\eta, g \xi)
    >
    \rmd(\rho, \eta) + \rmd(\eta, \xi)
    =
    \rmd(\rho, \xi)    
    \eqstop
  \end{equation*}
  Since the geodesics $[\rho, gy)$ and $[\rho, x)$ split in $g\xi$, we obtain $\rmm_\rho(x,gy) = \rmd(\rho, g\xi) > \rmd(\rho, \xi) = \rmm_\rho(x,y)$.  This finishes the proof of the proposition.

\end{proof}

The next lemma gives us many invariant neighbourhoods of points in $\partial T$.
\begin{lemma}
  \label{lem:compacts-fix-open-neighbourhoods}
  Let $K \leq \Aut(T)$ be a compact subgroup fixing $x \in \partial T$.  Then there is a basis of $K$-invariant neighbourhoods of $x$.
\end{lemma}
\begin{proof}
  Since $K$ is compact there is some $\rho \in T$ fixed by $K$.  Then $K$ fixes the geodesic $[\rho, x)$ pointwise.  So $K$ fixes elements of the neighbourhood basis of $x$
  \begin{equation*}
    \cU_x(n) = \{y \in \partial T \amid \rmm_\rho(x,y) \geq n\}
    \eqcomma
    n \in \NN
    \eqstop
  \end{equation*}
\end{proof}

\subsection{\Cstar-algebras and von Neumann algebras}
\label{sec:operator-algebras}

A \emph{\Cstar-algebra} is a Banach algebra $A$ such that $\|x^*x\| = \|x\|^2$ for all $x \in A$.  It is called simple if every closed *-ideal in $A$ is either trivial or equals $A$.

Denote by $\bo(H)$ the *-algebra of bounded operators on a Hilbert space $H$.  The topology of pointwise convergence on $\bo(H)$ is called \emph{strong topology}.  A \emph{von Neumann algebra} is a strongly closed unital *-subalgebra of $\bo(H)$.  Note that every von Neumann algebra is also a \Cstar-algebra.   Throughout the text we assume all von Neumann algebras to act on a separable Hilbert space.  A von Neumann algebra $M$ is simple if every strongly closed *-ideal of $M$ is either trivial or equals $M$.  Simple von Neumann algebras are called \emph{factors}.

Let $A$ be a \Cstar-algebra.  A \emph{projection} in $A$ is an element $p \in A$ satisfying $p = p^2 = p^*$.  If $p, q$ are projections, then we write $p \leq q$ in case $pq = p$.  A \emph{state} on $A$ is a unital functional $\vphi: A \ra \CC$ such that $\vphi(x^*x) \geq 0$ for all $x \in A$.  A state is \emph{tracial} if $\vphi(xy) = \vphi(yx)$ for all $x,y \in A$.

If $A \subset \bo(H)$ is a \Cstar-algebra, then its \emph{multiplier algebra}
\begin{equation*}
  \rM(A) := \{x \in \bo(H) \amid \forall y \in A: yx, xy \in A\}
\end{equation*}
is a \Cstar-algebra, independent of the representation of $A$ in $\bo(H)$.  It carries the \emph{strict topology} defined on nets by $x_\lambda \ra x$ if and only if $\|x_\lambda y - xy\|, \|y x_\lambda - yx\| \ra 0$ for all $y \in A$.  Since every \Cstar-algebra contains an approximate unit, it is a strictly dense two-sided ideal in its multiplier algebra.

\subsubsection{The type of a von Neumann algebra}
\label{sec:type}

Before defining the type of a von Neumann algebra, let us remark that every von Neumann algebra is the norm closure of the linear span of its projections.  So von Neumann algebras contain an abundance of projections.  The type of a von Neumann algebra depends on how its projections behave.
\begin{definition}
  \label{def:type}
  A factor $M$ is called finite if it admits a a tracial state; a projection $p \in M$ is finite if $pMp$ is finite.
  \begin{itemize}
  \item $M$ is of type ${\rm I}$ if it contains a minimal projection.
  \item $M$ is of type ${\rm II}$ if it contain a finite by not a minimal projection.
  \item In all other cases $M$ is of type ${\rm III}$.
  \end{itemize}
\end{definition}

In \cite{connes73-type-III} Connes introduced the invariant $\rS(M) \subset \RR_{\geq 0}$ of a factor $M$. He proved that $0 \notin \rS(M)$ if and only if $M$ is of type ${\rm I}$ or ${\rm II}$.  Further, $\rS(M) \cap \RR_{> 0}$ is a closed subgroup.  He then proceeds to the following definition.
\begin{definition}
  \label{def:type-III}
  Let $M$ be a factor of type ${\rm III}$.  If $\rS(M) = \{0,1\}$, then $M$ is of type ${\rm III}_0$.  If $\rS(M) = \lambda^\ZZ \cup \{0\}$ for some $\lambda \in (0,1)$, then we say that $M$ is of type ${\rm III}_\lambda$.  If $\rS(M) = \RR_{\geq 0}$, then $M$ is of type ${\rm III}_1$.  
\end{definition}

We need the following theorem to calculate the type of group von Neumann algebras appearing in this article.  We refer to Section \ref{sec:weight-theory} for a brief introduction to weights on von Neumann algebras. If $\vphi$ is a normal semi-finite faithful weight $\vphi$ on a von Neumann algebra $M$, we denote by $\Delta_\vphi$ its modular operator, by $(\sigma_t^\vphi)_t = (\Ad \Delta_\vphi^{it})_t$ the modular flow on $M$ and by $M^\vphi = \{ x \in M \amid \forall t: \, \sigma_t^\vphi(x) = x\}$ the fixed point algebra of the modular flow.  The spectrum of $\Delta_\vphi$ is denoted by $\sigma(\Delta_\vphi)$.
\begin{theorem}[{\cite[Corollary 3.2.7]{connes73-type-III}}]
  \label{thm:calculate-S}
  Let $M$ be a factor with a normal semi-finite faithful weight $\vphi$.  If $M^\vphi$ is a factor, then $\rS(M) = \sigma(\Delta_\vphi)$.
\end{theorem}

\subsection{Group \Cstar-algebras and group von Neumann algebras}
\label{sec:group-operator-algebras}

Let $G$ be a locally compact group and denote by $\lambda_G: G \ra \cU(\Ltwo(G))$ the left regular representation of $G$.  Then the \emph{group von Neumann algebra} of $G$ is $\rL(G) = \lambda_G(G)''$.  The canonical unitaries $\lambda_G(g)$, $g \in G$ are denoted $u_g \in \rL(G)$.   Fixing a left Haar measure $\mu$ of $G$, we define $\lambda_G:\contc(G) \ra \bo(\Ltwo(G))$ by 
\begin{equation*}
  \lambda_G(f) \xi(g) = \int_G f(h) \xi(h^{-1}g) \rmd \mu(h)
\end{equation*}
for all $f, g \in \contc(G)$.  The \emph{reduced group \Cstar-algebra} of $G$ is $\Cstarred(G) = \ol{\lambda_G(\contc(G))}^{\| \cdot \|}$.   The group \mbox{\Cstar-algebra} of $G$ does not contain the unitaries $u_g$ unless $G$ is discrete.  But $u_g \in \rM(\Cstarred(G))$ and even $u_g \contc(G) = \contc(G) = \contc(G) u_g$.  As a matter of fact, we have strongly dense inclusions $\Cstarred(G) \subset \rL(G)$ and $\rM(\Cstarred(G)) \subset \rL(G)$.

If $K \leq G$ is a compact open subgroup, we obtain an \emph{averaging projection} $p_K \in \contc(G) \subset \Cstarred(G)$ described by
\begin{equation*}
  p_K \xi(g) = \frac{1}{\mu(K)} \int_K \xi(kg) \rmd \mu(k)
\end{equation*}
for a square integrable function $\xi$ on $G$ representing an element of $\Ltwo(G)$.  These averaging projections play an important role in the present paper.  They form a natural approximate unit in $\Cstarred(G)$ and $\rL(G)$.  This is the content of the next proposition.
\begin{proposition}
  \label{prop:approximate-unit}
  Let $G$ be a totally disconnected group.  Then the net $(p_K)$, $K \leq G$ compact open subgroup, strictly converges to $1$ in $\Cstarred(G)$.  Further, it strongly converges to $1$ in $\rL(G)$.
\end{proposition}
\begin{proof}
  Since strict convergence implies strong convergence for bounded nets, it suffices to show that $(p_K)_K$ strictly converges to $1$ in $\Cstarred(G)$.  This in turn follows from a standard estimate for $\|(p_K x - x)\|_1$ for $x$ in the dense subalgebra $\contc(G) \subset \Cstarred(G)$ and $K$ some compact open subgroup of $G$.
  %
\end{proof}

For later use we want to note how averaging projections interact with each other and with the canonical unitaries $u_g$.  We start by describing relations between the averaging projections $p_K$ for different compact open subgroups $K \leq G$.  For the next statement, recall that $u_g \contc(G) = \contc(G)$.  Also note that $u_g p_L = p_L$ for all $g \in L$, so that the right hand side of the following equation is well-defined.
\begin{proposition}
  \label{prop:averaging-projections}
  Let $G$ be a locally compact group with compact open subgroups $L \leq K \leq G$.  In $\contc(G)$ we have
  \begin{equation*}
    p_K = \frac{1}{[K:L]} \sum_{gL \in K/L} u_g p_L
    \eqstop
    \end{equation*}
\end{proposition}
\begin{proof}
  Take $L \leq K \leq G$  as in the statement and let $f \in \contc(G)$ be arbitrary.  Let $\mu$ be a left Haar measure for $G$.  Then for all $h \in G$
\begin{align*}
  \bigl (\frac{1}{[K:L]}\sum_{gL \in K/L} u_g p_L f \bigr ) (h)
  & =
  \bigl (\frac{1}{[K:L]}\sum_{gL \in K/L}  p_L f \bigr ) (g^{-1}h) \\
  & = 
  \frac{1}{[K:L]}  \sum_{gL \in K/L} \frac{1}{\mu(L)} \int_L f(l^{-1}g^{-1}h) \rmd \mu(l) \\
  & =
  \frac{1}{\mu(K)} \sum_{gL \in K/L} \, \int_{g L} f(l^{-1}h) \rmd \mu(l) \\ 
  & =
  (p_K f )(h)
  \eqstop
\end{align*}
Since $\contc(G)$ is dense in $\Ltwo(G)$, we find that $p_K = \frac{1}{[K:L]} \sum_{gL \in K/L} u_g p_L$.
\end{proof}

The next lemma shows that averaging projections behave well with respect to conjugation by canonical unitaries.
\begin{lemma}
  \label{lem:conjugation-of-projections}
  Let $G$ be a locally compact group and $K \leq G$ be a compact open subgroup.  Then in $\contc(G)$ for all $g \in G$ we have
  \begin{equation*}
    u_g p_K u_g^* = p_{g K g^{-1}}
    \eqstop
  \end{equation*}
\end{lemma}
\begin{proof}
  Take $K \leq G$ as in the statement of the lemma and let $\mu$ be a left Haar measure for $G$.  For $f \in \contc(G)$ and  $g,h \in G$ we obtain
  \begin{align*}
    u_g p_K u_g^* f(h) 
    & = 
    \frac{1}{\mu(K)} \int_K f(g k^{-1} g^{-1} h) \rmd \mu(k) \\
    & =
    \frac{1}{\mu(K)} \Delta(g) \int_{g K g^{-1}} f(k^{-1}h) \rmd \mu(k) \\
    & =
    \frac{1}{\mu(g K g^{-1})} \int_{g K g^{-1}} f(k^{-1}h) \rmd \mu(k) \\
    & =
    p_{gKg^{-1}} f(h)
    \eqstop
  \end{align*}
  This shows that $u_g p_K u_g^* = p_{g K g^{-1}}$.
\end{proof}

We can next describe products of the form $p_K u_g p_K$ in $\contc(G)$.  The second part of the following proposition has a reformulation in terms of Hecke algebras (Section \ref{sec:hecke-cstar-algebras}).
\begin{proposition}
  \label{prop:multiplication-hecke-algebras}
  Let $K$ be a compact open subgroup and $g \in G$.  Put $L = K \cap g K g^{-1}$ Then 
  \begin{equation*}
    p_K u_g p_K = \frac{1}{[K : L]}\sum_{kL \in K/L} u_{kg} p_K
    \eqstop
  \end{equation*}
  In particular for $g, h \in G$ we have
  \begin{equation*}
    p_K u_h p_K u_g p_K = \frac{1}{[K : L]}\sum_{kL \in K/L} p_K u_{hkg} p_K
    \eqstop
  \end{equation*}
\end{proposition}
\begin{proof}
  By Proposition \ref{prop:averaging-projections}, we have 
  \begin{equation*}
    p_K = \frac{1}{[K : L]}\sum_{kL \in K/L} u_k p_L
    \eqstop
  \end{equation*}
  Using Lemma \ref{lem:conjugation-of-projections}, this implies
  \begin{equation*}
    p_K u_g p_K
    =
    \frac{1}{[K : L]} \sum_{kL \in K/L} \sum u_k p_L u_g p_K
    =
    \frac{1}{[K : L]}  \sum_{kL \in K/L} u_{kg} p_{g^{-1} L g}p_K
    =  
    \frac{1}{[K : L]}  \sum_{kL \in K/L} u_{kg} p_K
    \eqcomma
  \end{equation*}
  which proves the first part of the proposition.  The second part of the proposition follows directly from the first one.
\end{proof}

\subsubsection{Hecke \Cstar-algebras}
\label{sec:hecke-cstar-algebras}

Given a discrete Hecke pair $\Lambda \leq \Gamma$ there is a natural convolution product on double cosets.  Let $\CC(\Gamma, \Lambda)$ be the vector space whose basis consists of $v_g$, $\Lambda g \Lambda \in \Lambda \backslash \Gamma / \Lambda$.  We set $\rR(g) := [\Lambda : \Lambda \cap g \Lambda g^{-1}]$ and $\rL(g) := [\Lambda : \Lambda \cap g^{-1} \Lambda g]$.  We define a multiplication on $\CC(\Gamma, \Lambda)$ by
\begin{equation*}
    v_h v_g = \sum_{g' \Lambda \subset \Lambda g \Lambda} \frac{\rR(h)}{\rR(hg')} v_{hg'}
\end{equation*}
and an involution by
\begin{equation*}
  v_g^* = \frac{\rR(g)}{\rL(g)} v_{g^{-1}}
  \eqstop
\end{equation*}
There is a *-representation of $\CC(\Gamma, \Lambda)$ on $\ltwo(\Lambda \backslash \Gamma)$ via
\begin{equation*}
  v_h \delta_{\Lambda g}
  =
  \sum_{ \Lambda h' \subset \Lambda h \Lambda} \delta_{\Lambda h'g}
 \eqstop
\end{equation*}
The norm closure of $\CC(\Gamma, \Lambda)$ in this representation is the reduced Hecke-\Cstar-algebra of $\Lambda \leq \Gamma$, denoted by $\Cstarred(\Gamma, \Lambda)$.

In a completely analogous fashion, one associates with an inclusion $K \leq G$ of a compact open group into a locally compact group a reduced Hecke-\Cstar-algebra $\Cstarred(G,K)$.  Note that in this case the equality
\begin{equation*}
  \Delta(g) = \frac{[g^{-1} K g : K \cap g^{-1} K g]}{[K : K \cap g^{-1} K g]} = \frac{\rR(g)}{\rL(g)}
\end{equation*}
holds.  There is a *-isomorphism $p_K \Cstarred(G) p_K \cong \Cstarred(G,K)$ given by $p_K u_g p_K \mapsto  R(g)^{1/2} L(g)^{1/2} v_g$.   Moreover, if $K \leq G$ is the Schlichting completion of a discrete Hecke pair $\Lambda \leq \Gamma$, then $\Cstarred(G,K) \cong \Cstarred(\Gamma, \Lambda)$ as Tzanev showed in \mbox{\cite[Theorem 4.2]{tzanev03}}. 

\subsection{Weight theory}
\label{sec:weight-theory}

In this section we briefly recall weight theory and the Plancherel weight on the group \Cstar-algebra and the group von Neumann algebra of a locally compact group.  We refer the reader to \cite[Chapter VII]{takesaki03-II} for more details about weight theory on von Neumann algebras and \Cstar-algebras.  For a short summary of weights on \Cstar-algebras, we recommend \cite[Section 1]{kustermansvaes99}.  Weights should be thought of as integration against a possibly infinite measure on a noncommutative space.
\begin{definition}
  \label{def:weight}
  Let $A$ be a \Cstar-algebra.  A function $\vphi: A^+ \ra [0,\infty]$ is called a \emph{weight} if
\begin{itemize}
\item $\vphi(x + y) = \vphi(x) + \vphi(y)$ for all $x,y \in A^+$, and
\item $\vphi(rx) = r \vphi(x)$ for all $x \in A^+$, $r \in \RR_{\geq 0}$.
\end{itemize}

If $\vphi$ is a weight on $A$ then we denote by $\mathfrak n_\vphi := \{x \in A \amid \vphi(x^*x) < \infty\}$ the space of square integrable elements and by $\mathfrak m_\vphi = \mathfrak n_\vphi^* \mathfrak n_\vphi $ the space of integrable elements. Every element $x \in \mathfrak m_\vphi^+$ satisfies $\vphi(x) < \infty$.  There is a unique linear functional on $\mathfrak m_\vphi$ which extends $\vphi|_{\mathfrak m_\vphi^+}$.  We denote it also by $\vphi$.

$\vphi$ is called \emph{densely defined}, if $\mathfrak m_\vphi \subset A$ is dense.  We say that $\vphi$ is \emph{proper} if it is non-zero densely defined and lower semi-continuous in the norm topology.  A weight on a von Neumann algebra $M$ is called \emph{semi-finite} if $\mathfrak m_\vphi \subset M$ is strong-* dense.  It is called \emph{normal} if it is lower semi-continuous in the strong-* topology.  A normal semi-finite faithful weight is called an \emph{nsff weight}.
\end{definition}

\subsubsection{The GNS-construction}
\label{sec:GNS-construction}

\begin{definition}
  \label{def:gns-construction}
  Let $\vphi$ be a weight on a \Cstar-algebra $A$.  A GNS-construction for $\vphi$ is a triple $(H, \Lambda, \pi)$ where
  \begin{itemize}
  \item $H$ is a Hilbert space,
  \item $\Lambda: \mathfrak n_\vphi \ra H$ is a linear map with dense image such that $\langle \Lambda(x), \Lambda(y) \rangle = \vphi(y^*x)$ for all $x,y \in \mathfrak n_\vphi$.
  \item $\pi:A \ra \bo(H)$ is a *-representation satisfying $\pi(x)\Lambda(y) = \Lambda(xy)$ for all $a \in A$ and all $y \in \mathfrak n_\vphi$.
  \end{itemize}
\end{definition}

Every weight has a GNS-construction, which is unique up to unitary equivalence.

\subsubsection{The Plancherel weight and its modular automorphism group}
\label{sec:plancherel-weight}

  Let $G$ be a locally compact group and $\mu$ a left Haar measure on $G$.  Then the \emph{Plancherel weight} $\vphi$ on $\rL(G)$ satisfies $\vphi(f) = f(e)$ for all $f \in \contc(G) \subset \rL(G)$.  It is an nsff weight.  Its restriction to $\Cstarred(G)$ is a proper weight.  Note that a Plancherel weight depends on the choice of $\mu$ via the embedding $\contc(G) \subset \rL(G)$.  

If $K$ is a compact open subgroup, we associated the averaging projection $p_K = \frac{1}{\mu(K)} \int_K \lambda_k \rmd \mu(k)$ with it.  The Plancherel weight $\vphi$ satisfies 
\begin{equation*}
  \vphi(u_g p_K) =
  \begin{cases}
      \frac{1}{\mu(K)} &, g \in K \\
      0 &, g \in G \setminus K
      \eqstop
  \end{cases}
\end{equation*}
We will see in Lemma~\ref{lem:characterisation-plancherel-weight} that this property almost characterises Plancherel weights.

Plancherel weights are described in \cite[Chapter VII, \textsection 3]{takesaki03-II}.  The so called \emph{modular operator} of a Plancherel weight is the maximal self-adjoint positive multiplication operator associated with the modular function $\Delta$ of $G$.  For a Plancherel weight $\vphi$, the modular operator is denoted by $\Delta_\vphi$.  Its spectrum is $\sigma(\Delta_\vphi) = \ol{\Delta(G)}$.  Denote by $\sigma^\vphi_t = \Ad \Delta^{it}$ the so called \emph{modular automorphism group} of $\vphi$ (see \cite[Chapter VIII, \textsection 1]{takesaki03-II}.  It satisfies $\sigma^\vphi_t(u_g) = \Delta(g)^{it} u_g$ for all $g \in G$.  The set of elements $x \in \rL(G)$ such that the map $t \mapsto \sigma^\vphi_t(x)$ can be extended to an entire function on $\CC$ is called the set of analytic elements of $\vphi$.  All elements in $\contc(G)$ are analytic for any Plancherel weight on $\rL(G)$.

Let us collect some remarks about the Plancherel weight and its modular automorphism group.  
\begin{remark}
  \label{rem:restriction-of-modular-automorphism-group}
  Fix a locally compact group $G$ with modular function $\Delta$ and a left Haar measure $\mu$.  Let $\vphi$ be the Plancherel weight of $\rL(G)$ associated with $\mu$.
  \begin{itemize}
  \item Since $\sigma^\vphi_t$ preserves $\contc(G) \subset \rL(G)$, it restricts to a one-parameter group of *-automorphisms on $\Cstarred(G)$.  We refer to it as the natural one-parameter group of *-automorphism of the reduced group \Cstar-algebra.
  \item For a compact open subgroup $K \leq G$ we have $\Delta|_K \equiv 1$.  For any $z \in \CC$, we hence obtain 
    \begin{equation*}
      \sigma^\vphi_z(p_K)
      =
      \sigma^\vphi_z \bigl (\frac{1}{\mu(K)} \int_K u_g \rmd \mu (g) \bigr)
      =
      \frac{1}{\mu(K)} \int_K \sigma^\vphi_z(u_g) \rmd \mu (g)
      =
      p_K
      \eqstop
    \end{equation*}
  \end{itemize}
\end{remark}

\subsubsection{KMS-weights}
\label{sec:KMS-weights}

If $\vphi$ is a normal semi-finite faithful weight on a von Neumann algebra $M$, then it satisfies
\begin{equation*}
  \vphi(xy) = \vphi(y\sigma^\vphi_{-i}(x))
\end{equation*}
for all analytic square integrable elements $x,y \in M$. (Compare \cite[Chapter~VIII, \textsection~1, Definition~1.1]{takesaki03-II}).  A similar result for proper weights on \Cstar-algebras does not hold in general.  However, we are next going to define a class of weights on \Cstar-algebras which admit such control.  If $(\sigma_t)_t$ is a norm continuous one-parameter group of *-automorphisms, on a \Cstar-algebra $A$, then the set of analytic elements for $(\sigma_t)_t$ is dense in $A$ according to \cite[Section 1]{kustermans97}.  So the following definition makes sense.
\begin{definition}
  \label{def:KMS-weight}
  Let $A$ be a \Cstar-algebra and $(\sigma_t)_t$ a one-parameter group of *-automorphisms.  A~proper weight $\vphi$ on $A$ is called \emph{KMS-weight} with respect to $(\sigma_t)_t$ if
  \begin{itemize}
  \item $\vphi \circ \sigma_t = \vphi$ for all $t \in \RR$ and
  \item $\vphi(x^*x) = \vphi(\sigma_{\frac{i}{2}}(x)\sigma_{\frac{i}{2}}(x)^*)$ for all $x \in \rD(\sigma_{\frac{i}{2}})$,
  \end{itemize}
  where $\rD(\sigma_{\frac{i}{2}})$ denotes the domain of $\sigma_{\frac{i}{2}}$.
\end{definition}

Since a Plancherel weight on $\Cstarred(G)$ is a restriction of a Plancherel weight on $\rL(G)$, it is a KMS-weight.  For illustration, we explicitly work out the example of Plancherel weights of totally disconnected groups.
\begin{example}
  \label{ex:plancherel-KMS}
  Let $G$ be a totally disconnected locally compact group with left Haar measure $\mu$ and $\vphi$ be the Plancherel weight associated with $\mu$.  Then $\vphi$ is a KMS-weight with respect to the natural one-parameter group of *-automorphisms of $\Cstarred(G)$ from Remark \ref{rem:restriction-of-modular-automorphism-group}.

Indeed, let $g, h \in G$ and $K \leq G$ be a compact open subgroup.  Then by Proposition \ref{prop:multiplication-hecke-algebras}
\begin{equation*}
  \vphi(p_K u_g p_K u_h^* p_K)
  =
  \delta_{g,h} \frac{1}{\mu(K) [K : K \cap g^{-1} K g]}
  \eqstop
\end{equation*}
Also, using $\Delta(g) = \frac{\rR(g)}{\rL(g)}$, we have
\begin{align*}
  \vphi(\sigma_{i/2}(p_K u_h p_K)  \sigma_{i/2}(p_K u_g p_K)^*)
  & =
  \Delta(g)^{-1/2} \Delta(h)^{-1/2}  \, \vphi(p_K u_h p_K u_g^* p_K) \\
  & =
  \delta_{g,h} \Delta(g^{-1}) \frac{1}{\mu(K) [K : K \cap g^{-1}K g]} \\
  & =
  \delta_{g,h} \frac{1}{\mu(K) [K : K \cap g K g^{-1}]} \\
  & =
  \vphi(p_K u_g^* p_K u_h p_K)
  \eqstop
\end{align*}
Moreover,
\begin{equation*}
  \vphi(\sigma_t(p_K u_g p_K))
  =
  \Delta(g)^{it} \vphi(p_K u_g p_K)
  =
  \left \{
    \begin{array}{ll}
      \Delta(g)^{it} \frac{1}{\mu(K)} & \eqcomma \text{ if } g \in K \\
      0 & \eqcomma  \text{ if } g \in G \setminus K
    \end{array}
  \right \}
  =
  \vphi(p_K u_g p_K)
  \eqcomma
\end{equation*}
by the fact that $\Delta|_K \equiv 1$.  Since $\vphi$ is proper, \cite[Result 2.3 combined with Proposition 6.1]{kustermans97} applied to the subset $\bigcup_{\substack{K \leq G \\ \text{compact open}}} p_K \contc(G) p_K \subset \Cstarred(G)$ show that $\vphi$ is a \mbox{KMS-weight} with respect to $(\sigma_t)_t$.
\end{example}

One can show that the modular automorphism group of a KMS-weight, is implemented by a modular operator, which is described in the following proposition.
\begin{proposition}
  \label{prop:gns-and-modular-operator}
  Let $\psi$ be a KMS-weight with respect to a one-parameter group of \mbox{*-automorphisms} $(\sigma_t)_t$ on a \Cstar-algebra $A$.  Let $(H, \Lambda, \pi)$ be a GNS-construction for $\psi$.  There is a unique positive self-adjoint operator $\Delta_\psi$ on $H$ such that
  \begin{equation*}
    \Delta_\psi^{it} \Lambda(x) = \Lambda(\sigma_t(x))
  \end{equation*}
  for all $x \in \mathfrak n_\psi$.
\end{proposition}
We refer to \cite[Chapter VIII, §1, proof of Theorem 1.2]{takesaki03-II} and \cite[Section 2.2]{kustermansvaes99} for more details.

The notion of KMS-weights allows us to characterise the Plancherel weight on $\Cstarred(G)$ similar to the canonical trace on group \Cstar-algebras of discrete groups.  Recall that the natural one-parameter group of *-automorphisms of $\Cstarred(G)$ is the restriction of the modular flow of a Plancherel weight as described in Section \ref{sec:plancherel-weight}.
\begin{lemma}
  \label{lem:characterisation-plancherel-weight}
  Let $\psi$ be a KMS-weight for the natural one-parameter group of *-automorphisms on $\Cstarred(G)$.  If there is a left Haar measure $\mu$ on $G$ such that for every $g \in G$ and $K \leq G$ compact open we have
  \begin{equation*}
    \psi(u_g p_K) =
    \begin{cases}
      \frac{1}{\mu(K)} &, g \in K \\
      0 &, g \notin K
      \eqcomma
    \end{cases}
  \end{equation*}
  then $\psi$ is the Plancherel weight associated with $\mu$.
\end{lemma}
We refer to the poofs of Proposition 1.14 and Corollary 1.15 in \cite{kustermansvaes99}.  Denoting by $\vphi$ the Plancherel weight normalised to $\vphi(p_K) = \psi(p_K)$, then the proofs given by Kustermans-Vaes can be taken over, if we observe that that $p_K$, $K \leq G$ compact open, is a net converging to $1$ strictly in $\Cstarred(G)$ by Proposition \ref{prop:approximate-unit} and $\psi$ and $\vphi$ agree on elements of the form $ap_K$.

The following proposition can be found for example in \cite[Proposition 1.13]{kustermansvaes99}.  It says that we have similar control over KMS-weights on a \Cstar-algebra as we have over normal semi-finite faithful weights on a von Neumann algebra.
\begin{proposition}
  \label{prop:KMS-exchange-property}
  Let $\psi$ be a KMS-weight to with respect to $(\sigma_t)_t$ on a \Cstar-algebra $A$.  Denote by $\cA$ the analytic subalgebra of $(\sigma_t)_t$.  If $x,y \in \cA$ are square integrable with respect to $\psi$, then $\psi(xy) = \psi(y \sigma_{-i}(x))$.
\end{proposition}

\subsubsection{The $\rS$-invariant of a group von Neumann algebra}
\label{sec:S-invariant-group-von-Neumann-algebra}

We explain how to determine Connes' $\rS$-invariant (Section \ref{sec:type}) for group von Neumann algebras of totally disconnected groups.  For the purpose of this paper, it suffices to analyse the modular operator of a Plancherel weight.

We start by identifying the centraliser (see \cite[Chapter VIII, \textsection 2]{takesaki03-II}) of a Plancherel weight.  If $\vphi$ is an nsff weight on a von Neumann algebra $M$, and $(\sigma^\vphi_t)_t$ is the modular automorphism group of $\vphi$, then the fixed point algebra of $(\sigma^\vphi_t)_t$ is denoted by $M^\vphi$.  It is called the centraliser of $\vphi$.  For the next proposition recall from Section \ref{sec:totoally-disconnected-groups} that $G_0$ denotes the kernel of the modular function of a locally compact group $G$.
\begin{proposition}
  \label{prop:centraliser-plancherel-weight}
  Let $G$ be a totally disconnected group and $\vphi$ a Plancherel weight on $\rL(G)$.  Then $\rL(G)^\vphi = \rL(G_0)$..
\end{proposition}
\begin{proof}
  Since $\sigma_t^\vphi(u_g) = \Delta(g)^{it} u_g$ for all $g \in G$, it follows that $\rL(G_0) \subset \rL(G)^\vphi$.  We prove the converse inclusion.  For $x \in \rL(G)^\vphi$ and a compact open subgroup $K \leq G$, we have $xp_K \in \rL(G)^\vphi$.  By Proposition~\ref{prop:approximate-unit} it suffices to prove that $\rL(G)^\vphi p_K \subset \rL(G_0)$.  Let $x \in \rL(G)^\vphi p_K$.  Since $p_K \in \mathfrak{n}_\vphi$, we can consider $\Lambda_\vphi(x) = \sum_{gK \in G/K} x_{g} \mathbb{1}_{gK}$ for unique scalars $x_g \in \CC$.  By Proposition \ref{prop:gns-and-modular-operator}, we have
  \begin{equation*}
    \Lambda_\vphi(x)
    =
    \Lambda_\vphi(\sigma^\vphi_t(x))
    =
    \Delta_\vphi^{it} \Lambda_\vphi(x)
    =
    \sum_{gK \in G/K} x_{gK} \Delta(g)^{it} \mathbb{1}_{gK}
    \eqstop
  \end{equation*}
  By uniqueness of the coefficients $x_g$, $gK \in G/K$, we see that $\Lambda_\vphi(x) \in \Ltwo(G_0)$.  This shows that $x \in \rL(G_0)$, which finishes the proof.
\end{proof}

\begin{proposition}
  \label{prop:spectrum-plancherel-weight}
  Let $G$ be a locally compact group and $\vphi$ be a Plancherel weight on $\rL(G)$.  Then $\sigma(\Delta_\vphi) = \ol{\Delta(G)}$.
\end{proposition}
\begin{proof}
  We saw in Section \ref{sec:plancherel-weight} that $\Delta_\vphi$ is the multiplication operator associated with $g \mapsto \Delta(g)$.  So the proposition follows right away.
\end{proof}

\begin{theorem}
  \label{thm:S-invariant-group-von-Neumann-algebra}
  Let $G$ be a totally disconnected group such that $\rL(G_0)$ is a factor.  Then $\rS(\rL(G)) = \ol{\Delta(G)}$.
\end{theorem}
\begin{proof}
  By Proposition \ref{prop:centraliser-plancherel-weight} we may apply Theorem \ref{thm:calculate-S} to a Plancherel weight $\vphi$ of $\rL(G)$.  Then Proposition \ref{prop:spectrum-plancherel-weight} shows that $\rS(\rL(G)) = \sigma(\Delta_\vphi) = \ol{\Delta(G)}$.
\end{proof}

\section{Groups acting on trees with open stabilisers of boundary points}
\label{sec:open-stabilisers}

In this section we describe different aspects of groups $G$ acting on a tree $T$ such that $G_x \leq G$ is open for some point $x \in \partial T$.  In Theorem \ref{thm:characterisation-condition-star} we describe condition ($*$) from the introduction, which applies to all groups treated in the rest of this article.
\begin{proposition}
  \label{prop:characaterisation-open-stabiliser}
  Let $G \leq \Aut(T)$ be a closed subgroup and let $x \in \partial T$.  Then the following statements are equivalent.
  \begin{itemize}
  \item $G_x \leq G$ is open.
  \item For every compact open subgroup $K \leq G$ the orbit $Kx$ is finite.
  \item There is a compact open subgroup $K \leq G$ fixing $x$.
  \end{itemize}
\end{proposition}
\begin{proof}
  If $G_x$ is open, then $|K x | = [K : K \cap G_x]$ is finite for all compact open subgroups $K \leq G$.  Further, if $|Kx| < \infty$ for some compact open subgroup, then $K_x \leq K$ is closed and has finite index and it is hence a compact open subgroup of $G$.  Finally, if some compact open subgroup $K$ fixes $x$, then $K \leq G_x$, showing that $G_x$ is open.
\end{proof}

\begin{proposition}
  \label{prop:characterisation-fixed-point-with-open-stabiliser}
  Let $G \leq \Aut(T)$ be a closed subgroup.  Let $g \in G$ be hyperbolic and denote by $x$ the attracting fixed point of $g$.  Then the following statements are equivalent.
  \begin{itemize}
  \item $G_x \leq G$ is open.
  \item There is a compact open subgroup $K \leq G$ such that $g^n K g^{-n} \geq K$ for all $n \in \NN$.
  \item  There is a compact open subgroup $K \leq G$ such that $[K : K \cap g^n K g^{-n}]$, $n \in \NN$, is bounded.
  \item  For all compact open subgroups $K \leq G$ the sequence $[K : K \cap g^n K g^{-n}]$, $n \in \NN$, is bounded.
  \end{itemize}
\end{proposition}
\begin{proof}
  Assume that $G_x \leq G$ is open.  Let $\rho$ be a vertex of $T$ on the axis of $g$.  Then $K := G_x \cap G_\rho = G_{[\rho, x)}$ is a compact open subgroup of $G$.  We have $g^n K g^{-n} = G_{[g^n \rho, x)} \geq G_{[\rho, x)} = K$ for all $n \in \NN$.

  Assume that there is $K \leq G$ such that $g^n K g^{-n} \geq K$ for all $n \in \NN$.  Then $K \cap g^n K g^{-n} = K$ and hence $[K : K \cap g^n K g^{-n}] = 1$ is bounded in $n$.

  Assume that there is a compact open subgroup $L \leq G$ such that $[L : L \cap g^n L g^{-n}]$ is bounded in $n$.  Let $K \leq G$ be a compact open subgroup.  Then
  \begin{align*}
    [K : K \cap g^n K g^{-n}]
    & \leq
    [K : K \cap g^n (K \cap L) g^{-n}] \\
    & \leq
    [K : K \cap g^n L g^{-n}] [L : K \cap L] \\
    & \leq
    [(K \cap L) : (K \cap L) \cap g^n L g^{-n}]  [K : K \cap L] [L :  K \cap L]  \\
    & \leq
    [L : L \cap g^n L g^{-n}]   [K : K \cap L] [L :  K \cap L]
    \eqstop
  \end{align*}
  It follows that $[K : K \cap g^n K g^{-n}]$ is bounded in $n$.

  Assume that for all compact open subgroups $K \leq G$ the sequence $[K : K \cap g^n K g^{-n}]$, $n \in \NN$ is bounded.  Let $\rho$ be a vertex on the axis of $g$.  Then $K := G_\rho$ is a compact open subgroup and $K \cap g^n K g^{-n} = G_{[\rho, g^n \rho]}$ is a descending sequence of compact open subgroups.  Since $[K : K \cap g^n K g^{-n}]$ is bounded in $n$, the sequence $K \cap g^n K g^{-n}$ becomes stationary.  So $\bigcap_{n \geq 0} K \cap g^n K g^{-n}$ is an open subgroup.  Then also  $G_x \geq G_{[\rho, x)} = \bigcap_{n \geq 0} K \cap g^n K g^{-n}$ is open.
\end{proof}

Before we describe hyperbolic elements both of whose fixed points have an open stabiliser, we need to note the following well-known lemma.
\begin{lemma}
  \label{lem:strucutre-of-point-stabilisers}
  Let $G \leq \Aut(T)$ be a closed subgroup fixing a point $x \in \partial T$.  Let $H$ be the set of all elliptic elements of $G$.  Then $H$ is an ascending union of compact open subgroups and every element of $G \setminus H$ is hyperbolic.
  For every $g \in G \setminus H$ whose attracting fixed point is $x$ there is a compact open subgroup $K \leq H$ such that $gKg^{-1} \geq K$ and $H = \bigcup_{n \in \ZZ} g^n K g^{-n}$.
\end{lemma}
\begin{proof}
  Let $g,h \in G$ be elliptic. There are $\rho, \eta \in T$ fixed by $g$ and $h$ respectively.  Let $\xi$ be a point in $[\rho, x) \cap [\eta, x)$.  Then $g$ and $h$ fix $\xi$, so $gh$ fixes $\xi$ and it is hence elliptic.  Since $H$ is the union of vertex stabilisers, it is open in $G$.  Since $G$ fixes $x \in \partial T$, it consists only of elliptic and hyperbolic elements.  Hence $G \setminus H$ consists entirely of hyperbolic elements.

Now let $g \in G \setminus H$ be a hyperbolic whose attracting fixed point is $X$.  Let $\rho$ be a vertex on the axis of $g$ and $K = G_{[\rho,x)}$.  Then $g K g^{-1} = G_{[g \rho, x)} \geq K$.  If $k \in H$, then there is $\rho' \in T$ fixed by $k$.  Since $k x = x$, $k$ fixes all points on the geodesic $[\rho', x)$.  So $k \in G_{[\rho', x) \cap [\rho, x)} \subset \bigcup_{n \in \ZZ} g^n K g^{-n}$.
\end{proof}

\begin{proposition}
  \label{prop:characterisation-both-fixed-points-with-open-stabiliser}
  Let $G \leq \Aut(T)$ be a closed subgroup and $g \in G$ hyperbolic such that one of its fixed points in $\partial T$ has an open stabiliser in $G$.  Then the following statements are equivalent.
  \begin{itemize}
  \item Both fixed points of $g$ have an open stabiliser in $G$.
  \item $g$ lies in the kernel $G_0$ of the modular function of $G$.
  \item $g$ normalises a compact open subgroup of $G$.
  \end{itemize}
  In particular, if $g \in G_0$ is hyperbolic, then either both or none of its fixed points have an open stabiliser.
\end{proposition}
\begin{proof}
  Throughout the proof, $g \in G$ denotes a hyperbolic element such that $G_x$ is open for one of its fixed points $x \in \partial T$.  Replacing $g$ by its inverse if necessary, we may assume that its attracting fixed point has an open stabiliser in $G$.

  By Proposition \ref{prop:characterisation-fixed-point-with-open-stabiliser}, there is a compact open subgroup $K \leq G$ such that $g^n K g^{-n} \geq K$ for all $n \in \NN$.  Then
  \begin{equation*}
    [K : K \cap g^{-n} K g^n] = [g^n K g^{-n} : K] = \Delta(g^{-n}) = \Delta(g)^{-n}
    \eqstop
  \end{equation*}
  
  If the repelling fixed point of $g$ has an open stabiliser, Proposition \ref{prop:characterisation-fixed-point-with-open-stabiliser} and the previous equation show that $\Delta_G(g)^{-1}$ is a positive integer, whose powers are bounded.  Hence $\Delta_G(g) = 1$, which proves that $g \in G_0$.

  If $g \in G_0$, let $K \leq G$ be a compact open subgroup such that $g K g^{-1} \geq K$, which is provided by Lemma \ref{lem:strucutre-of-point-stabilisers}.  Since left and right multiplication with $g$ preserve the left Haar measure of $G$, this implies $g K g^{-1} = K$.  So $g$ normalises a compact open subgroup of $G$.

  Finally if $g$ normalises some compact open subgroup $K \leq G$, then Proposition \ref{prop:characterisation-fixed-point-with-open-stabiliser} applies to $g$ and $g^{-1}$, showing that both fixed points of $g$ have an open stabiliser.

  Since an arbitrary hyperbolic element $g \in G$ normalises a compact open subgroup of $G$ if and only if $g^{-1}$ does so, the last statement of the proposition follows.
\end{proof}

Let us now describe condition ($*$) as it is mentioned in the introduction.  To this end, we need a characterisation of amenable subgroups of $\Aut(T)$.
\begin{proposition}[{Adams-Ballmann \cite{adamsballmann98}}]
  \label{prop:adams-ballmann}
  Let $G \leq \Aut(T)$ be a closed subgroup. 
  \begin{itemize}
  \item If $G$ is amenable, then it fixes some point in $V(T) \cup E(T) \cup \partial T$.  In case $G$ fixes an edge of $T$, then it contains an index 2 subgroup fixing a vertex of $T$.
  \item If $G$ fixes a point in $V(T) \cup E(T) \cup \partial T$, then it is amenable.
  \end{itemize}
\end{proposition}

\begin{theorem}
  \label{thm:characterisation-condition-star}
  Let $G$ be a topological group.  Then the following two conditions on $G$ are equivalent.
  \begin{itemize}
  \item There is a locally finite tree $T$ such that $G \leq \Aut(T)$ as a closed subgroup.  Further, $G$ is non-amenable without any non-trivial compact normal subgroup and there is a compact open subgroup $K \leq G$ such that $\cN_G(K)/K$ contains an element of infinite order.
  \item There is a locally finite thick tree $T$ such that $G \leq \Aut(T)$ as a closed subgroup.  Further, $G$ acts minimally on $\partial T$ and there is some point $x \in \partial T$ such that $G_x \leq G$ is open and $G_x \cap G_0$ contains a hyperbolic element.
  \end{itemize}
\end{theorem}
\begin{proof}
  Assume that $G$ satisfies our first condition and take $G \leq \Aut(T)$ as in the statement.   Since $G$ is non-amenable there is a minimal $G$-invariant subtree $T' \leq T$ by Proposition \ref{prop:existence-minimal-subtree}.  Note that $T'$ must be thick.  Since $G$ does not contain any compact normal subgroups, we can consider $G \leq \Aut(T')$ by restriction.  Since $G \grpaction{} T'$ admits no proper $G$-invariant subtree, it follows that $G \grpaction{} \partial T'$ is minimal by Proposition \ref{prop:equivalence-minimal-action}.  Let $K \leq G$ be a compact open subgroup and $g \in \cN_G(K)$ an element whose image in $\cN_G(K)/K$ has infinite order.  Then $H := \langle K , g \rangle \leq G$ is a non-compact amenable open subgroup.  So Proposition \ref{prop:adams-ballmann} says that there is $x \in \partial T$ fixed by $H$.  In particular, $H \leq G_x \leq G$ is open.  Moreover, $g \in G_x$ is not contained in any compact subgroup, so $g$ is hyperbolic.  Now $G \leq \Aut(T')$ satisfies all conditions of the second statement.

Now assume that $G$ satisfies the second condition of the theorem and take $G \leq \Aut(T)$ as in the statement.  Note that $G$ is not compact, since it contains a hyperbolic element.  Since $T$ is thick and $G \grpaction{} \partial T$ is minimal, it follows that $G$ is non-amenable.  Moreover, $G \grpaction{} T$ is minimal by Proposition \ref{prop:equivalence-minimal-action}.  So any compact normal subgroup of $G$ fixes $T$ pointwise, showing that it is trivial.  Take $x \in \partial T$ such that $G_x$ is open and $G_x \cap G_0$ contains a hyperbolic element $g$.  By Propositions~\ref{prop:characterisation-fixed-point-with-open-stabiliser}~and~\ref{prop:characterisation-both-fixed-points-with-open-stabiliser} there is a compact open subgroup $K \leq G$ such that $g \in \cN_G(K)$.  Since $g \in G_0$ it follows that $g \in \cN_G(K)$.  Since $g$ is hyperbolic, it is not contained in any compact subgroup of $G$ and hence its image in $\cN_G(K)/K$ has infinite order.  So $G$ satisfies all conditions of the first statement.  This finishes the proof of the theorem.
\end{proof}

\section{Locally compact Powers groups}
\label{sec:locally-compact-powers-groups}

In this section we introduce the notion of a locally compact Powers group.  This generalises an idea of Powers \cite{powers75} and de la Harpe \cite{delaharpe85} to prove \Cstar-simplicity for discrete groups.  We prove an analogue of Powers averaging.  Further we adapt an idea of de la Harpe and Pr{\'e}aux \cite{delaharpepreaux11} used to show \Cstar-simplicity of HNN-extensions in order to prove that a natural class of groups acting on trees are Powers groups.  This justifies our definition of Powers groups in the context this article.  We however insist on the ad hoc character of the next definition.

\begin{definition}[Locally compact Powers group]
  \label{def:lc-powers-group}
  Let $G$ be a locally compact group.  Let $K \leq G$ be a compact open subgroup, $F \subset G \setminus K$ be a compact set and $r \in \NN$.  We say that $G$ satisfies \emph{Powers property} with control $r$ with respect to $K$ and $F$  if the following condition holds.

For all $n \in \NN^\times$ there are elements $g_1, \dotsc, g_n \in G_0$  and a decomposition of $G$ into left $K$-invariant sets $G = C \sqcup D$ such that
  \begin{itemize}
  \item for all $f \in F$ we have $f C \cap C = \emptyset$,
  \item the sets $g_1D, \dotsc, g_n D$ are pairwise disjoint, and
  \item $[K : K \cap g_i K g_i^{-1}] \leq r$ for all $i \in \{1, \dotsc, n\}$.
  \end{itemize}
\end{definition}
If we do not need to specify control, we speak of Powers property with respect to $K$ and $F$ only.

\begin{remark}
  \begin{itemize}
  \item A discrete group $G$ is a Powers group in the sense of \cite{delaharpe85} if and only if it has the Powers property with respect to $\{e\}$ and $F$ for all finite sets $F \subset G \setminus \{e\}$.
  \item For the results of this paper it is important to introduce explicit control over the subgroup $K$ and the constant $r$ in our formulation of Powers property.  This becomes clear from Proposition~\ref{prop:lc-powers-group} and the proof of Theorem \ref{thm:cstar-simplicity}.
  \item We may replace the control condition \mbox{``$[K : K \cap g_i K g_i^{-1}] \leq r$ for all $i \in \{1, \dotsc, n\}$''} by  \mbox{``$[K : K \cap g_i^{-1} K g_i] \leq r$ for all $i \in \{1, \dotsc, n\}$''}, since $g_1, \dotsc, g_n \in G_0$. Indeed, $[K : K \cap g K g^{-1}] = \Delta(g) [K : K \cap g^{-1} K g]$ for all $g \in G$.
  \end{itemize}
\end{remark}

The next proposition generalises Powers averaging to locally compact Powers groups.  We adapt the proof given in \cite{delaharpe07-survey}.
\begin{proposition}
  \label{prop:powers-averaging}
  Let $G$ be a locally compact group.  Assume that $G$ has the Powers property with control $r \in \NN$ with respect to the compact open subgroup $K \leq G$ and the compact set $F \subset G \setminus K$ .  Then for all $x \in \contc(G)$ whose support lies in $F$ and all  $\veps > 0$ there is $n \in \NN$ and elements $g_1, \dotsc, g_n \in G_0$ satisfying
  \begin{equation*}
    \| \frac{1}{n} \sum_{i = 1}^n u_{g_i} (x p_K) u_{g_i}^* \| < \veps
    \eqcomma
  \end{equation*}
  and $[K : K \cap g_i K g_i^{-1}] \leq r$ for all $i \in \{1, \dotsc, n\}$.
\end{proposition}
\begin{proof}
  Take $K \leq G$, $F \subset G \setminus K$ and $r \in \NN$ as in the statement of the proposition.   Checking the definition of Powers property, we see that we may assume $F = F K$.  Let $x \in \contc(G)$ have support in $F$ and let $\veps > 0$.  For all $n \in \NN^\times$ there are elements $g_1, \dotsc, g_n \in G_0$ and a decomposition of $G$ into $K$-invariant sets $G = C \sqcup D$ such that
  \begin{itemize}
  \item for all $f \in F \setminus K$ we have $f C \cap C = \emptyset$,
  \item the sets $g_1D, \dotsc, g_n D$ are pairwise disjoint, and
  \item $[K : K \cap g_i K g_i^{-1}] \leq r$ for all $i \in \{1, \dotsc, n\}$.
  \end{itemize}
  Let $\tilde F$ be a set of representatives for $F / K$ and write
  \begin{equation*}
    x p_K = \sum_{f \in \tilde F} x_f u_f p_K
  \end{equation*}
  for some scalars $x_f \in \CC$, $f \in \tilde F$.  Since $C$ is $K$-invariant, it is open.  For $i \in \{1, \dotsc, n\}$ we may consider the orthogonal projections
\begin{equation*}
  q_i: \Ltwo(G) \ra \Ltwo(g_iC)
  \eqstop
\end{equation*}
Then for all $f \in F$
\begin{equation*}
  u_f p_K u_{g_i}^* q_i \Ltwo(G)
  =
  \Ltwo(f C)
  \qquad \text{ and } \qquad
  u_{g_i}^* q_i \Ltwo(G)
  =
  \Ltwo(C)
  \eqstop
\end{equation*}
So
\begin{equation*}
  \langle q_i u_{g_i}  u_f p_K u_{g_i}^* q_i \Ltwo(G), \Ltwo(G) \rangle
  =
  \langle u_f p_K u_{g_i}^* q_i \Ltwo(G), u_{g_i}^* q_i \Ltwo(G) \rangle
  =
  \langle \Ltwo(fC), \Ltwo(C) \rangle
  =
  \{ 0 \}
  \eqstop
\end{equation*}
This shows $q_i u_{g_i}  u_f p_K u_{g_i}^* q_i = 0$.  Moreover, the images of $(1 - q_i) u_{g_i}  u_f p_K u_{g_i}^*$, $i \in \{1,\dotsc, n\}$ are pairwise orthogonal.  Also, the supports of $q_i u_{g_i}  u_f p_K u_{g_i}^* (1 - q_i)$, $i \in \{1, \dotsc, n\}$ are pairwise orthogonal.  We can proceed to the following estimate.
\begin{align*}
  & \quad 
  \| \frac{1}{n} \sum_{i = 1}^n u_{g_i} (xp_K )u_{g_i}^* \| \\
  & \leq
  \frac{1}{n} \| \sum_{i = 1}^n (1- q_i) u_{g_i} x p_K u_{g_i}^* \| +
  \frac{1}{n} \| \sum_{i = 1}^n q_i u_{g_i} x p_K u_{g_i}^* (1 - q_i) \| \\
  & \leq 
  \frac{\sqrt{n}}{n} \Bigl(
  \sup \bigl \{\| (1-q_i)u_{g_i} x p_K u_{g_i}^*\| \amid i \in \{1, \dotsc, n\} \bigr \} +
  \sup \bigl  \{\| q_i u_{g_i} x p_K u_{g_i}^* (1 - q_i)\| \amid i \in \{1, \dotsc, n\} \bigr \}
  \Bigr ) \\
  & \leq
  \frac{2\sqrt{n}}{n}   \|x\|
  \eqstop
\end{align*}
Taking $n$ big enough so that $\frac{2 \sqrt{n}}{n}\|x\| < \veps$ finishes the proof.
\end{proof}

In the rest of the section we are going to prove that groups satisfying condition ($*$) have the Powers property with respect to specific compact open subgroups.  To this end, we need an abundance of pairwise transverse hyperbolic elements with uniform control over how well they commensurate compact open subgroups. 
\begin{lemma}
  \label{lem:many-hyperbolic-elements}
  Let $G \leq \Aut(T)$ be a closed non-amenable subgroup acting minimally on $\partial T$.  Assume that there is $x \in \partial T$ such that $G_x$ is open and $G_x \cap G_0$ contains a hyperbolic element.  Then for every non-empty open set $\cO \subset \partial T$ there is $r \in \NN$ and a sequence of pairwise transverse hyperbolic elements $(g_i)_{i \geq 1}$ in $G_0$ such that the fixed points of $g_i$ lie in $\cO$ for all $i \geq 1$ and for all compact open subgroups $K \leq G$ the indices $[K : K \cap g_i^l K g_i^{-l}]$ are uniformly bounded in $i \geq 1$ and $l \in \ZZ$.
\end{lemma}
\begin{proof}
 Let $\cO \subset \partial T$ be open.  Since $G_{gx} = g G_x g^{-1}$ for all $x \in \partial T$ and all $g \in G$, the set of points $x \in \partial T$ such that $G_x$ is open and $G_x \cap G_0$ contains a hyperbolic element is dense by minimality of $G \grpaction{} \partial T$.  So we may take $h \in G_0$ hyperbolic with attracting fixed point in $\cO$.  Since $G$ is not amenable, we may then find a conjugate $g$ of $h$ which is transverse to $h$.

Replacing $h$ by some positive power of itself, we may assume that $h$ maps the fixed points of $g$ into $\cO$.   Then the elements $g_i := h^i g h^{-i}$ , $i \in \NN^\times$ are hyperbolic and their fixed points lie in $\cO$.

Take a compact open subgroup $K \leq G$.  For all $i \geq 1$ and $l \in \ZZ$ we have 
\begin{align*}
  [K : K \cap g_i^l K g_i^{-l}]
  & = 
  [K : K \cap h^i g^l h^{-i} K (h^i g^l h^{-i})^{-1}] \\
  & =
  [h^{-i} K h^i : h^{-i} K h^i \cap g^l h^{-i} K h^i g^{-l}] \\
  & \leq
  [h^{-i} K h^i : h^{-i} K h^i \cap g^l (K \cap h^{-i} K h^i ) g^{-l}] \\
  & \leq
  [h^{-i} K h^i : h^{-i} K h^i \cap g^l K g^{-l}] [K : K \cap h^{-i} K h^i] \\
  & \leq
  [(K \cap h^{-i} K h^i) : (K \cap h^{-i} K h^i) \cap g^l K g^{-l}]  [K : K \cap h^{-i} K h^i] [h^{-i} K h^i : K \cap h^{-i} K h^i] \\
  & \leq
  [K : K \cap g^l K g^{-l}] [K : K \cap h^{-i} K h^i] [K : K \cap h^i K h^{-i}]
\end{align*}
Since by Proposition \ref{prop:characterisation-both-fixed-points-with-open-stabiliser} the fixed points of $g, h \in G_0$ have open stabilisers, Proposition \ref{prop:characterisation-fixed-point-with-open-stabiliser} implies that $[K : K \cap g_i^l K g_i^{-l}]$ is uniformly bounded for $i \geq 1$ and $l \in \ZZ$.  Possibly passing to a subsequence of $(g_i)_{i \geq 1}$, we may assume that these elements are pairwise transverse, which finishes the proof.
\end{proof}

We proceed to show Powers property for groups satisfying condition ($*$).  The proof is inspired by Proposition 8 of \cite{delaharpepreaux11}.
\begin{proposition}
  \label{prop:lc-powers-group}
  Let $G$ be a group satisfying condition ($*$).  Then the following statements hold true.
  \begin{enumerate}
  \item \label{it:lc-powers-group:F-fixed}
    There is $r \in \NN$ and a compact open subgroup $K \leq G$ such that for all compact sets $F \subset G \setminus K$, $G$ has the Powers property with control $r$ with respect to $(K, F)$.
  \item \label{it:lc-powers-group:stabilisers}
    For a compact open subgroup $L \leq G$ and $x \in \partial G$ such that $G_x$ is open, set $K := L \cap G_x$.  Then $G$ has the Powers property with respect to $(K, L \setminus K)$.
  \item \label{it:lc-powers-group:K-restricted}
    For all neighbourhoods of the identity $N \subset G$ there is a compact open subgroup $K \leq G$ contained in $N$ such that $G$ has the Powers property with respect to $(K,gK)$ for all $g \in G \setminus K$.
  \end{enumerate}
\end{proposition}
\begin{proof}
  We start to consider case \ref{it:lc-powers-group:F-fixed}.  Fix $x \in \partial T$ such that $G_x \leq G$ is open and $G_x \cap G_0$ contains a hyperbolic element.  By Lemma~\ref{lem:strucutre-of-point-stabilisers} and Proposition \ref{prop:characterisation-both-fixed-points-with-open-stabiliser} we find a maximal compact open subgroup $K \leq G_x$ such that $G_x \setminus K$ consists of hyperbolic elements only.  Let $F \subset G \setminus K$ be a compact set.  We construct an open $K$-invariant subset $\cO \subset \partial T$ such that $f \cO \cap \cO = \emptyset$ for all $f \in F$.  Consider $F \setminus G_x$.  It is covered by finitely many right $K$-cosets.  Further $f x \neq x$ for all $f \in F \setminus G_x$.  By Lemma~\ref{lem:compacts-fix-open-neighbourhoods}, we find an open $K$-invariant set $\cO_0 \ni x$ such that $(F \setminus G_x) \cO_0 \cap \cO_0 = \emptyset$.  Since $F \subset G \setminus K$, Lemma \ref{lem:strucutre-of-point-stabilisers} applies to show that $F \cap G_x$ only contains hyperbolic elements fixing $x$.  It is covered by cosets  $h_1K, \dotsc, h_mK$ where $h_1, \dotsc, h_m$ are hyperbolic elements fixing $x$.  Find $\rho \in T$ fixed by $K$.  Since $x$ is a fixed point for all hyperbolic elements $h_1, \dots, h_m$, we can apply Proposition \ref{prop:contractive} to $\rmm_\rho(x, \cdot \,)$.  We find some $d \in \NN$ such that $\cO_1 := \{ z \in \partial T \amid \rmm_\rho(x,z) = d\}$ satisfies $h_i \cO_1 \cap \cO_1 = \emptyset$ for all $i \in \{1, \dotsc, m\}$.  Since $K$ fixes $x$ and $\rmm_\rho$ is $K$-invariant, we see that $k \cO_1 = \cO_1$ for all $k \in K$. So $(F \cap G_x) \cO_1 \cap \cO_1 = \emptyset$.  Recall from Remark~\ref{rem:control-shadow-topology} that the sets $\{z \in \partial T \amid \rmm_\rho(x,z) \geq d'\}$, $d' \in \NN$, form a basis of compact open neighbourhoods of $x$.  We hence may choose $d$ big enough so as to assume that $\cO_0 \cap \cO_1 \neq \emptyset$.  Putting $\cO := \cO_1 \cap \cO_0$, we found a non-empty $K$-invariant open subset of $\partial T$ such that $f \cO \cap \cO = \emptyset$ for all $f \in F$.

Let $g \in G_0$ be some hyperbolic element whose attracting fixed point lies in $\cO$ and whose repelling fixed point we denote by $y \in \partial T$.  Let $\cU \ni y$ be a clopen neighbourhood.  By Lemma \ref{lem:many-hyperbolic-elements} there is a sequence of pairwise transverse hyperbolic elements $(h_i)_{i \geq 1}$ in $G_0$ whose fixed points lie in $\partial T \setminus (\{x\} \cup \cU)$.  Also, we find $r_0 \in \NN$ such that ${[K : K \cap h_i^l K h_i^{-l}]} \leq r_0$ for all $i \geq 1$ and all $l \in \ZZ$.  Note that $r$ is independent of $F$.    Denote the attracting fixed point of $h_i$ by $\alpha_i$ and its repelling fixed point by $\omega_i$.  Since $\alpha_i, \omega_i \in \partial T \setminus (\{x\} \cup \cU)$, $i \geq 1$, there is some $m \in \NN$ such that $g^m \alpha_i, g^m \omega_i \in \cO$ for all $i \geq 1$.  Let $g_i := g^m h_i g^{-m}$.  Now calculation of Lemma \ref{lem:many-hyperbolic-elements} shows that 
\begin{equation*}
  [K : K \cap g_i^l K g_i^{-l}] \leq [K : K \cap g^m K g^{-m}] r_0^2  =: r
\end{equation*}
for all $i \geq 1$ and all $l \in \ZZ$.  Before we finish the proof in case \ref{it:lc-powers-group:F-fixed}, let us consider the other cases of the proposition.

Consider case \ref{it:lc-powers-group:stabilisers}.   Take a compact open subgroup $L \leq G$ and $x \in \partial T$ such that $G_x$ is open.  Put $K := L \cap G_x$ and $F := L \setminus K$.  Since $L$ is compact and $G_x$ is open, $Fx$ is finite by Proposition \ref{prop:characaterisation-open-stabiliser}.  Moreover, $x \notin Fx$.  By Proposition \ref{lem:compacts-fix-open-neighbourhoods} there is an open $K$-invariant set $\cO \subset \partial T$ such that $f \cO \cap \cO = \emptyset$ for all $f \in F$.

Let us now consider case \ref{it:lc-powers-group:K-restricted}.  Take a neighbourhood of the identity $N \subset G$.  Note that for any set $\Sigma \subset \partial T$ containing at least three different points, $G_\Sigma$ fixes a point in $T$ and it is hence compact.  Since $G \grpaction{} \partial T$ is faithful, in any dense set $\Sigma_0 \subset \partial T$ we find a finite subset $\Sigma \subset \Sigma_0$ such that $G_\Sigma \subset N$.  The set
\begin{equation*}
\Sigma_0 := \{ x \in \partial T \amid G_x \text{ is open and } G_x \cap G_0 \text{ contains a hyperbolic element}\}  
\end{equation*}
is dense, because $G \grpaction{} \partial T$ is minimal.  It follows that there is a finite subset $\Sigma \subset \Sigma_0$ such that $G_\Sigma \subset N$.  Put $K := G_\Sigma$.  If $g \in G \setminus K$, then there is $x \in \Sigma$ such that $g x \neq x$.  Let $\cO$ be an open neighbourhood of $x$ such that $g \cO \cap \cO = \emptyset$.  Since $K$ fixes $x$, Proposition \ref{lem:compacts-fix-open-neighbourhoods} says that we may make $\cO$ smaller so as to assume that it is $K$-invariant.  For the last part of the proof we set $F := gK$.

In all cases \ref{it:lc-powers-group:F-fixed},\ref{it:lc-powers-group:stabilisers} and \ref{it:lc-powers-group:K-restricted}, we have a compact open subgroup $K \leq G$ and a compact right $K$-invariant set $F \subset G \setminus K$ together with an open $K$-invariant set $\cO \subset \partial T$ satisfying $f \cO \cap \cO = \emptyset$ for all $f \in F$.  In case \ref{it:lc-powers-group:F-fixed}, $K$ is independent of the choice of $F$.  Moreover, in case \ref{it:lc-powers-group:F-fixed} we already constructed a sequence $(g_i)_{i \geq 1}$ of pairwise transverse hyperbolic elements whose fixed points lie in $\cO$.  There is some $r \in \NN$ independent of $F$ such that $[K : K \cap g_i^l K g_i^{-l}] \leq r$ for all $i \geq 1$ and all $l \in \ZZ$.  In cases \ref{it:lc-powers-group:stabilisers} and \ref{it:lc-powers-group:K-restricted}, we may choose such a sequence $(g_i)_{i \geq 1}$ according to Lemma \ref{lem:many-hyperbolic-elements}.  But in the latter case, $r$ possibly depends on $F$.  For the next paragraph, denote by $\omega_i$ the attracting fixed point of $g_i$, $i \geq 1$.

Now pick $x \in \partial T$ and set
\begin{gather*}
  C := \{g \in G \amid g x \in \cO\} \eqcomma \\
  D := \{g \in G \amid g x \notin \cO\}
  \eqstop
\end{gather*}
Then $C$ and $D$ are left $K$-invariant, since $\cO$ is $K$-invariant.  Moreover, $f\cO \cap \cO = \emptyset$ implies $f C \cap C = \emptyset$ for all $f \in F$.  Fix $n \in \NN$.  We can find pairwise disjoint open neighbourhoods of $W_i$ of $\omega_i$, $i \in \{1, \dotsc, n\}$ and exponents $l_i \in \NN^\times$ such that \mbox{$g_i^{l_i} (\partial T \setminus \cO) \subset W_i$}.  Replacing each $g_i$ by $g_i^{l_i}$, we may assume that $g_i (\partial T \setminus \cO) \subset W_i$.  Then $g_1D, \dotsc, g_n D$ are pairwise disjoint sets.  This finishes the proof of the proposition.
\end{proof}

 \section{Fullness of averaging projections}
\label{sec:fullness-of-pK}

In this section we take a first step to prove our \Cstar-simplicity result.  If a group $G$ is \Cstar-simple, then in particular the projections $p_K$ averaging over a compact open subgroup of $G$ are full in $\Cstarred(G)$.  We are not aware of any simple criterion ensuring fullness of $p_K$.  So the aim of this section is to prove that averaging projections in reduced group \Cstar-algebras of groups satisfying condition ($*$) are full. 

We start with a lemma ensuring invertibility of certain averages in $\Cstarred(G)$ in the proof of Proposition \ref{prop:pK-full-projection}.
\begin{lemma}
  \label{lem:invertible-average}
  Let $G$ be a locally compact group and $K \leq G$ a compact open subgroup.  Then for all $g \in G$ we have
  \begin{equation*}
    p_K u_g^* p_K  u_g p_K \geq [K: K \cap g K g^{-1}]^{-2} p_K
    \eqstop
  \end{equation*}
\end{lemma}
\begin{proof}
  Take a compact open subgroup $K \leq G$ and $g \in G$.  It suffices to prove that for all $\xi \in p_K \Ltwo(G)$, we have $\langle p_K u_g^* p_K  u_g p_k \xi, \xi \rangle \geq [K: K \cap g K g^{-1}]^{-2} \| \xi \|^2$. Let $\mu$ be the left Haar measure for $G$ satisfying $\mu(K) = 1$.  For $\xi  = \sum_h \xi_h \mathbb{1}_{Kh} \in p_K \Ltwo(G)$, we have
  \begin{multline*}
    \langle p_K u_g^* p_K  u_g p_k \xi, \xi \rangle
    =
    \|p_K u_g \xi \|^2
    =
    \sum_h \int_G  |\xi_h|^2 \bigl | (p_K \mathbb{1}_{gKh})(l) \bigr |^2 \rmd \mu(l) \\
    =
    \sum_h \int_G  |\xi_h|^2 \bigl | (\int_K \mathbb{1}_{gKh})(k^{-1}l) \rmd \mu (k) \bigr |^2 \rmd \mu (l)
    \eqstop
  \end{multline*}
  Since $k^{-1}l \in gKh$ if and only if $k \in lh^{-1} K g^{-1}$, we obtain for $l \in g K h$ and $k \in gKg^{-1} \cap K$ that $\mathbb{1}_{gKh}(k^{-1}l) = 1$.  We can hence continue the previous equation and obtain
  \begin{align*}
    \langle p_K u_g^* p_K  u_g p_k \xi, \xi \rangle
    & \geq
    \sum_h |\xi_h|^2 \mu(gKh) \mu(gKg^{-1} \cap K)^2 \\
    & =
    \mu(gKg^{-1} \cap K)^2 \sum_h |\xi_h|^2 \mu(Kh)  \\
    & =
    [K: K \cap g K g^{-1}]^{-2} \|\xi\|^2
    \eqstop
  \end{align*}
  This finishes the proof of the lemma.
\end{proof}

\begin{proposition}
  \label{prop:pK-full-projection}
  Assume that $G$ is a group satisfying condition ($*$).  Then $p_K$ is a full projection in $\Cstarred(G)$, i.e. the closed two-sided ideal generated by $p_K$ equals $\Cstarred(G)$.
\end{proposition}
\begin{proof}
  We have to show that for all compact open subgroups $K \leq G$ the closed two-sided ideal generated by $p_K$ equals $\Cstarred(G)$.  We start by proving the following claim.

\textbf{Claim}.  Let $I \unlhd \Cstarred(G)$ be a closed two-sided ideal such that $p_K \in I$ for some compact open subgroup $K \leq G$.  Then $p_{K \cap G_x} \in I$ for all $x \in \partial T$ that have an open stabiliser $G_x \leq G$.  \\
Fix $I \unlhd \Cstarred(G)$, $K \leq G$ and $x \in \partial T$ as in the claim.  By Proposition~\ref{prop:lc-powers-group} the group $G$ has the Powers property with respect to $L := K \cap G_x$ and $F := K \setminus L$.  We can write $[K:L] p_K = \sum_{gL \in K/L} u_g p_L$, as shown by Proposition~\ref{prop:averaging-projections}.  So Proposition \ref{prop:powers-averaging} says that there is $r \in \NN$ such that for all $\veps > 0$ there are $g_1, \dotsc, g_n \in G_0$ satisfying $[L : L \cap g_i L g_i^{-1}] \leq r$ and 
\begin{equation*}
  \|\frac{1}{n} \sum_{i = 1}^n u_{g_i} ([K:L] p_K - p_L) u_{g_i}^*\| < \veps
  \eqstop
\end{equation*}
Note that $[L : L \cap g_i^{-1} L g_i] = [L : L \cap g_i L g_i^{-1}] \Delta(g_i^{-1}) \leq r$ for all $i \in \{1, \dotsc, n\}$.  By Lemma \ref{lem:invertible-average} we have 
\begin{equation*}
  p_L u_{g_i} p_L u_{g_i}^* p_L \geq  [L : L \cap g_i^{-1} L g_i]^{-2}p_L \geq  r^{-2} p_L
\end{equation*}
for all $i \in \{1, \dotsc, n\}$.  We showed that there is $0 < \delta := r^{-2} < 1$ such that for all $\veps > 0$ there is $x :=  \frac{[K:L]}{n} \sum_{i = 1}^n p_L u_{g_i}  p_K u_{g_i}^* p_L\in I $ and an invertible element $y := \frac{1}{n} \sum_{i = 1}^n p_L u_{g_i} p_L u_{g_i}^* p_L \in p_L \Cstarred(G)p_L$ such that $\|x - y\| < \veps$ and $\sigma(y) \subset [\delta, 1]$.  Then $\sigma(y^{-1}) \subset [1, \delta^{-1}]$ and in particular $\|y^{-1}\| \leq \delta^{-1}$.  So
  \begin{equation*}
    \|x  \cdot y^{-1} - p_L\|
    \leq \|x - y\| \|y^{-1}\|
    < \veps \delta^{-1}
    \eqstop
  \end{equation*}
  Choosing $\veps \leq \delta$, we conclude that $x  \cdot y^{-1} \in I$ is invertible in $p_L \Cstarred(G)p_L$.  So $p_L \in I$, which proves the claim.

Since $G \grpaction{} \partial T$ is minimal the set
\begin{equation*}
  \Sigma_0 := \{ x \in \partial T \amid G_x \text{ is open and } G_x \cap G_0 \text{ contains a hyperbolic element}\}
\end{equation*}
is dense in $\partial T$.  Moreover, for every set $\Sigma \subset \Sigma_0$ such that $|\Sigma| \geq 3$, the pointwise stabiliser $G_\Sigma$ fixes a point in $T$ and is hence compact.  By faithfulness of $G \grpaction{} \partial T$, we conclude that the compact open subgroups $G_\Sigma$ with $\Sigma \subset \Sigma_0$ and $|\Sigma| \geq 3$ form a neighbourhood basis of $e$ in $G$.  So the claim combined with Proposition~\ref{prop:averaging-projections} shows that $p_L \in \ol{\Cstarred(G) p_K \Cstarred(G)}$ for all compact open subgroups $L \leq G$.

By Proposition \ref{prop:approximate-unit}, $(p_K)$, $K \leq G$ compact open, strictly converges to $1$ in $\Cstarred(G)$, it follows that the closed two-sided ideal generated by all projections $p_K$, $K \leq G$ compact open subgroup, is $\Cstarred(G)$.  This finishes the proof of the proposition.
\end{proof}

\section{\Cstar-simplicity}
\label{sec:cstar-simplicity}

This section has two aims.  First, we show that a \Cstar-simple group must be totally disconnected, extending Proposition~4 of \cite{bekkacowlingdelaharpe94}.  Then we combine results from Sections~\ref{sec:locally-compact-powers-groups}~and~\ref{sec:fullness-of-pK} in order to prove that groups satisfying condition ($*$) are \mbox{\Cstar-simple}.  In connection with examples from Section~\ref{sec:BS-groups}, this gives rise to the first \Cstar-simplicity result for non-discrete groups.

Let $G$ be a locally compact group and $\pi$ a unitary representation of $G$.  We denote by $\tilde \pi$ the *-representation of the maximal group \Cstar-algebra $\Cstarmax(G)$ that is induced by $\pi$.  If $\pi, \rho$ are two unitary representations of $G$, then we say that $\pi$ is weakly contained in $\rho$ and write $\pi \prec \rho$, if $\ker \tilde \pi \supset \ker \tilde \rho$.  If $\pi \prec \rho$ and $\pi \succ \rho$, we say that $\pi$ and $\rho$ are weakly equivalent and write $\pi \sim \rho$.  We refer to \cite[Appendix F]{bekkadelaharpevalette08} for more a good summary of basic properties of weak containment.  It is clear from the definitions that $G$ is \Cstar-simple if and only if $\pi \prec \lambda_G$ implies $\pi \sim \lambda_G$ for every unitary representation $\pi$ of $G$.  We will use this characterisation of \Cstar-simplicity in the proof of the following theorem.  
\begin{theorem}
  \label{thm:cstar-simple-implies-td}
  Let $G$ be a locally compact \Cstar-simple group.  Then $G$ is totally disconnected.
\end{theorem}
\begin{proof}
  Let $G$ be a locally compact group that is not totally disconnected.  We have to show that there is a unitary representation $\pi \prec \lambda_G$ such that $\pi \not \sim \lambda_G$.

Since $G$ is not totally disconnected, the connected component $G^0 \leq G$ of the identity is not trivial.  Let $K \unlhd G^0$ the maximal compact normal subgroup.  By the structure theorem for locally compact groups, $G^0/K$ is a connected Lie group.  If $R$ denotes the inverse image in $G^0$ of the amenable radical of $G^0/K$, then $H := G^0/R$ is a connected semi-simple Lie group with trivial centre and hence it is linear.  Note that $R$ is amenable, since it is compact-by-amenable.  Moreover, $R$ is a characteristic subgroup of $G^0$ and hence normal in $G$.  If $R \neq \{e\}$ then the quasi-regular representation $\lambda_{G, R}$ is weakly contained in $\lambda_G$, but it is not injective on $G$.  Hence $\lambda_{G, R} \not \sim \lambda_G$, which finishes the proof

So we may assume that $R = \{e\}$ and hence $H = G^0$ is a connected semi-simple linear Lie group.  Let $\pi_0$ be a principal series representation of $H$ and let $\pi = \mathrm{Ind}_H^G(\pi_0)$ be the induced representation of $G$.  Then $\pi \prec \lambda_G$, since $\pi_0 \prec \lambda_H$.  We show that $\pi$ is not weakly equivalent to $\lambda_G$.  By Mackey's subgroup theorem \cite[Theorem 12.1]{mackey55-induced-representations-I}, we have $\mathrm{Res}_H^G (\pi) \sim \bigoplus_{\alpha \in G} \pi_0 \circ \alpha \prec \bigoplus_{\alpha \in \Aut(H)} \pi_0 \circ \alpha$, where we apply $\alpha \in G$ via its image in $\Aut(H)$.
%
%
%

By the definition of the Fell topology, $\{ \rho \text{ irrep of } H \amid \rho \prec \pi_0\}$ is closed in the unitary dual $\hat H$.  Since every semi-simple Lie group has a finite outer automorphism group, by \cite[Corollary 2]{murakami52}, and $\pi_0 \circ \alpha$ depends up to unitary equivalence only on the image of $\alpha \in \Aut(H)$ in $\Out(H)$, it follows that $\bigcup_{\alpha \in \Aut(H)} \{ \rho \text{ irrep of } H \amid \rho \prec \pi_0 \circ \alpha\}$ is closed in $\hat H$. So if $\rho$ is an irreducible unitary representation of $H$ such that $\rho \prec \bigoplus_{\alpha \in \Aut(H)} \pi_0 \circ \alpha$, then $\rho \prec \pi_0 \circ \alpha$ for some $\alpha \in \Aut(H)$.  Since $H$ is a CCR group by \cite[Theorem 15.5.6]{dixmier77} and since $\pi_0 \circ \alpha$ is irreducible, we conclude that $\rho \cong \pi_0 \circ \alpha$.

We showed that $\mathrm{Res}_H^G(\pi)$ weakly contains only finitely many irreducible representations of $H$. However $\lambda_H = \mathrm{Res}_H^G(\lambda_G)$ weakly contains all principal series representations, of which there are infinitely many.  This implies $\pi \not \sim \lambda_G$ and finishes the proof of the theorem.
%
\end{proof}

\begin{theorem}
  \label{thm:cstar-simplicity}
  Let $G$ be a group satisfying condition ($*$).  Then $\Cstarred(G)$ is simple.
\end{theorem}
\begin{proof}
  Take $G \leq \Aut(T)$ as in the statement of the theorem.  Let $I \lhd \Cstarred(G)$ be a non-trivial closed two-sided ideal and take $0 \neq x \in I$ positive.  By Proposition \ref{prop:lc-powers-group} \ref{it:lc-powers-group:F-fixed} there is a compact open subgroup $K \leq G$ and some number $r \in \NN$ such that for all compact subsets $F \subset G \setminus K$, Powers property with control $r$ holds with respect to $(K, F)$ inside $G$.  We choose the Plancherel weight $\vphi$ on $\Cstarred(G)$ satisfying $\vphi(p_K) = 1$.  We may scale $x$  so that $\vphi(p_K x p_K) = 1$.  Fix $\veps > 0$ and find $y_0 \in \contc(G)$ such that $\|x - y_0\| < \frac{\veps}{2}$.  Since $G$ has the Powers property with control $r$ with respect to $K$ and $\supp y_0 \setminus K$, Proposition \ref{prop:powers-averaging} says that there are elements $g_1, \dotsc, g_n \in G_0$ for which Powers averaging gives
  \begin{equation*}
    \| \frac{1}{n} \sum_{i = 1}^n u_{g_i} (y_0p_K - p_K ) u_{g_i}^* \| < \frac{\veps}{2}
  \end{equation*}
  and $[K : K \cap g_i K g_i^{-1}] \leq r$ for all $i \in \{1, \dotsc, n\}$.  We obtain that
\begin{align*}
  & \quad 
  \| \frac{1}{n} \sum_{i = 1}^n u_{g_i} (x - p_K )u_{g_i}^* \| \\
  & \leq
  \| \frac{1}{n} \sum_{i = 1}^n u_{g_i} (y_0  - p_K )u_{g_i}^* \| + \frac{\veps}{2} \\
  & < \frac{\veps}{2} + \frac{\veps}{2}
  \eqstop
\end{align*}
Then also 
\begin{equation*}
  \| \frac{1}{n} \sum_{i = 1}^n p_K u_{g_i} x u_{g_i}^* p_K - \frac{1}{n}\sum_{i = 1}^n p_K u_{g_i} p_K u_{g_i}^* p_K \| < \veps
  \eqstop  
\end{equation*}

Note that $[K : K \cap g_i^{-1} K g_i]  = [K : K \cap g_i K g_i^{-1}] \Delta(g_i^{-1}) \leq r$ for all $i \in \{1, \dotsc, n\}$.  By Lemma \ref{lem:invertible-average}, we have $p_K u_g p_K  u_g^* p_K \geq [K: K \cap g^{-1} K g]^{-2} p_K$ for all $g \in G$.  This implies that $\frac{1}{n} \sum_{i = 1}^n p_K u_{g_i} p_K u_{g_i}^* p_K \geq r^{-2} p_K$.

 Summarising the proof up to now, we showed that there is $0 < \delta : = r^{-2} < 1$ such that for all $\veps > 0$ there is $a := \frac{1}{n} \sum_{i = 1}^n p_K u_{g_i} x u_{g_i}^* p_K \in I $ and an invertible element $b := \frac{1}{n} \sum_{i = 1}^n p_K u_{g_i} p_K u_{g_i}^* p_K \in p_K\Cstarred(G)p_K$ such that $\|a - b\| < \veps$ and $\sigma(b) \subset [\delta, 1]$.  Then $\sigma(b^{-1}) \subset [1, \delta^{-1}]$ and in particular $\|b^{-1}\| \leq \delta^{-1}$.  So
  \begin{equation*}
    \|a  \cdot y^{-1}  - p_K\|
    \leq \|a - b\| \|b^{-1}\|
    < \veps \delta^{-1}
    \eqstop
  \end{equation*}
  Choosing $\veps \leq \delta$, we conclude that $a \cdot b^{-1}  \in I$ is invertible in $p_K \Cstarred(G)p_K$.  So $p_K \in I$.  By Lemma~\ref{prop:pK-full-projection}, we obtain $I = \Cstarred(G)$.  This finishes the proof.
\end{proof}

\subsection{Applications}
\label{sec:applications}

In this section we apply Theorem \ref{thm:cstar-simplicity} to two problems of independent interest.  We first give to the best of our knowledge the first non-trivial examples of simple reduced Hecke-\Cstar-algebras and then show that certain groups acting on trees are not of type ${\rm I}$.
\begin{corollary}
  \label{cor:simple-Hecke-algebra}
  Let $T$ be thick tree and $\Gamma \leq \Aut(T)$ some not necessarily closed group acting without proper invariant subtree. Let $\Lambda$ be some vertex stabiliser of in $\Gamma$ and assume that there is a finite index subgroup $\Lambda_0 \leq \Lambda$ such that $\cN_\Gamma(\Lambda_0)/\Lambda_0$ contains an element of infinite order.  Then $\Cstarred(\Gamma, \Lambda)$ is simple.
\end{corollary}
\begin{proof}
  Take $\Lambda = \Gamma_\rho \leq \Gamma \leq \Aut(T)$ as in the statement of the corollary.  Let $G = \ol{\Gamma}$ and $K = \ol{\Lambda}$.  Then the natural bijection between $G/K$ and $\Gamma/\Lambda$ conjugates $G$ and $\Gamma \sslash \Lambda$.  So by \cite[Theorem~4.2]{tzanev03}, we have $\Cstarred(\Gamma, \Lambda) \cong p_K \Cstarred(G) p_K$.  In view of Theorem \ref{thm:cstar-simplicity}, it hence suffices to show that $G$ satisfies condition ($*$).

Since $\Gamma$ acts on $T$ without any proper minimal subtree, also $G$ does so.  So Proposition \ref{prop:equivalence-minimal-action} shows that $G \grpaction{} \partial T$ is minimal.  Since $T$ is thick, $\partial T$ is a Cantor space.  So $G$ does not fix a point in $\rV(T) \cup \rE(T) \cup \partial T$.  Now Proposition \ref{prop:adams-ballmann} implies that $G$ is not amenable.  The closure of any finite index subgroup of $\Lambda$ is open in $K$.  So the closure $L$ of $\Lambda_0$ is a compact open subgroup and $\cN_G(L)/L = \cN_\Gamma(\Lambda_0)/\Lambda_0$ contains an element of infinite order.  We verified condition ($*$), finishing the proof of the corollary.
\end{proof}

\begin{corollary}
  \label{cor:not-type-I}
  Let $T$ be a thick tree and $G \leq \Aut(T)$ be a closed subgroup acting minimally on $\partial T$.  Assume that there is $x \in \partial T$ such that 
  \begin{itemize}
  \item $Kx$ is finite for some compact open subgroup $K \leq G$, and
  \item there is some hyperbolic element in $G_0 \cap G_x$.
  \end{itemize}
  Then $G$ is not a type ${\rm I}$ group.
\end{corollary}
\begin{proof}
  Take $G \leq \Aut(T)$ as in the statement of the corollary and assume that $G$.  Then by Proposition \ref{prop:characaterisation-open-stabiliser}, $G$ satisfies property (*) and hence also $G_0$ satisfies (*).  Assuming that $G$ is a type ${\rm I}$ group, also its open subgroup $G_0$ is a type ${\rm I}$ group.  So we may assume in addition that $G$ is unimodular.

  By Theorem \ref{thm:cstar-simplicity} $G$ is \Cstar-simple.  Since it is also a type ${\rm I}$ group, $\lambda$ is unitarily equivalent to a multiple of an irreducible representation.  In particular, $\rL(G) \cong \bo(H)$ for some Hilbert space $H$.  Since $G$ is unimodular, the Plancherel weight $\vphi$ on $\rL(G)$ agrees with the unique tracial weight on $\bo(H)$.  So if $K \leq G$ denotes some compact open subgroup of $G$, the fact that $\vphi(p_K) < \infty$, implies that  $p_k$ is a finite projection in $\bo(H)$.  This shows that $p_K \rL(G) p_K$ is isomorphic to a finite type ${\rm I}$ factor, i.e. $p_K \rL(G) p_K \cong \rM_n(\CC)$ for some $n \in \NN$.  Take some hyperbolic element $g \in G_x$.  Then by $p_k u_g p_K$ satisfies $\langle p_K u_{g^n} p_K \xi, \eta \rangle \ra 0$ for all $\xi, \eta \in p_K \Ltwo(G)$.  This means that $p_K u_{g^n} p_K$ converges to $0$ weakly.  However, $p_k u_{g^n}^* p_K u_{g^n} p_K \geq [K : K \cap g^n K g^{-n}]^{-2} p_K$, which is bounded from below by Proposition \ref{prop:characterisation-fixed-point-with-open-stabiliser}.   So $p_K u_g p_K$ does not converge strongly to $0$, contradicting the isomorphism $p_K \rL(G) p_K \cong \rM_n(\CC)$.  This finishes the proof of the corollary.
\end{proof}

\section{KMS-weights and von Neumann factors}
\label{sec:KMS-weights-and-factors}

In this section we apply Powers group methods to prove uniqueness of certain natural KMS-weights on the reduced group \Cstar-algebra of certain groups $G$ satisfying condition ($*$).  This uniqueness allows us to conclude factoriality of the associated group von Neumann algebra $\rL(G)$ and to determine its type.

\begin{theorem}
  \label{thm:unique-KMS-state}
  Let $G$ be a group satisfying condition ($*$) and assume that some compact open subgroup of $G$ is topologically finitely generated.  Let $\vphi$ be a KMS-weight for the natural one-parameter group $(\sigma_t)_t$ on $\Cstarred(G)$.  Assume that $\vphi(p_K) < \infty$ for all compact open subgroups $K \leq G$.  Then $\vphi$ is a Plancherel weight on $\Cstarred(G)$.
\end{theorem}
\begin{proof}
  Let $\vphi$ be a KMS-weight as described in the statement of the theorem.  By Lemma \ref{lem:characterisation-plancherel-weight} it suffices to show that there is a left Haar measure $\mu$ on $G$ such that for all compact open subgroups $K \leq G$ and for all $g \in G$ we have
  \begin{equation*}
    \vphi(u_g p_K) =
    \begin{cases}
      \frac{1}{\mu(K)} & \eqcomma g \in K \\
      0 &\eqcomma \text{otherwise.}
    \end{cases}
  \end{equation*}

\textbf{Claim}. If $K \leq G$ is a compact open subgroup and $g \in G \setminus K$, then $\vphi(u_g p_K) = 0$. \\
Using Proposition \ref{prop:averaging-projections} it suffices to show the claim for a neighbourhood base of $e$ consisting of compact open subgroups $K \leq G$.  So fix a neighbourhood $e \in N \subset G$.  By Proposition \ref{prop:lc-powers-group} we find a compact open subgroup $K \leq G$ that is contained in $N$ such that $G$ has the Powers property with respect to $K$ and $gK$ for all $g \in G \setminus K$.  Proposition \ref{prop:powers-averaging} says that there is $r \in \NN$ such that for every $\veps > 0$ there are elements $g_1, \dots g_n \in G_0$ such that
\begin{equation*}
  \|\frac{1}{n} \sum_{i = 1}^n u_{g_i} u_g p_K u_{g_i}^*\| < \veps
\end{equation*}
and $[K : K \cap g_i K g_i^{-1}] \leq r$ for all $i \in \{1, \dotsc, n\}$.  Using a GNS-construction associated with $\vphi$, this gives rise to the following estimate for any averaging projection $p_L$ with $L \leq G$ compact open.
\begin{equation*}
  \vphi(p_L \frac{1}{n} \sum_{i = 1}^n u_{g_i} u_g p_K u_{g_i}^*p_L)
  =
  \langle \frac{1}{n} \sum_{i = 1}^n u_{g_i} u_g p_K u_{g_i}^* p_L, p_L \rangle_\vphi
  \leq
  \|\frac{1}{n} \sum_{i = 1}^n u_{g_i} u_g p_K u_{g_i}^*\| \|p_L\|_{2, \vphi}^2
  \leq \veps \vphi(p_L)^2
  \eqstop
\end{equation*}
Since $K$ is topologically finitely generated by Proposition \ref{prop:characterisation-top-finitely-generated}, Lemma~\ref{lem:control-number-subgroups} says that there are only finitely many subgroups in $K$ that have index bounded by $r$.  The intersection $L$ of all closed subgroups of $K$ which have index bounded by $r$ satisfies $g_i^{-1} L g_i \leq K$ for all $i \in \{1, \dotsc, n\}$.  Lemma \ref{lem:conjugation-of-projections} implies that $p_K u_{g_i}^* p_L = p_K u_{g_i}^*$ for all $i \in \{1, \dotsc, n\}$.  Using the KMS-condition with Proposition~\ref{prop:KMS-exchange-property} and $\Delta(g_i) = 1$ for all $i \in \{1, \dotsc, n\}$ we see that
\begin{align*}
  \veps \vphi(p_L)^2 
  & \geq \vphi(p_L \frac{1}{n} \sum_{i = 1}^n u_{g_i} u_g p_K u_{g_i}^* p_L) \\
  & = \vphi(\frac{1}{n} \sum_{i = 1}^n u_{g_i} u_g p_K u_{g_i}^*) \\
  & = \frac{1}{n} \sum_{i = 1}^n  \vphi(u_g p_K u_{g_i}^* \sigma_{-i}(u_{g_i}) ) \\
  & = \frac{1}{n} \sum_{i = 1}^n  \Delta(g_i) \vphi(u_g p_K u_{g_i}^* u_{g_i} ) \\
  & = \vphi(u_g p_K)
  \eqstop
\end{align*}
Since $\veps > 0$ is arbitrary and the choice of $L$ is independent of $\veps$, we see that $\vphi(u_g p_K) = 0$.  This proves the claim.

Fix a compact open subgroup $L \leq G$ and let $\mu$ be the left Haar measure of $G$ satisfying $\mu(L) = \frac{1}{\vphi(p_L)}$.  If $K \leq G$ is an arbitrary compact open subgroup, we can apply Proposition \ref{prop:averaging-projections} and the claim to obtain
\begin{align*}
  \frac{1}{\mu(L)}
  & =
  \vphi(p_L) \\
  & =
  \frac{1}{[L : K \cap L]} \sum_{g (K \cap L) \in L / K \cap L} \vphi(u_g p_{K \cap L}) \\
  & =
  \frac{1}{[L : K \cap L]} \vphi(p_{K \cap L}) \\
  & = 
  \frac{1}{[L : K \cap L]} \sum_{g (K \cap L)  \in K / K \cap L} \vphi(u_g p_{K \cap L}) \\
  & =
  \frac{[K : K \cap L]}{[L : K \cap L]} \vphi(p_K)
  \eqstop
\end{align*}
We infer that $\vphi(p_K) = \frac{1}{\mu(K)}$, finishing the proof of the proposition.
\end{proof}

Since a Plancherel weight on $\rL(G)$ restricts to the corresponding Plancherel weight on $\Cstarred(G)$, the previous theorem allows us to conclude factoriality of $\rL(G)$.  We are also able to compute its type.
\begin{theorem}
  \label{thm:factoriality}
  Let $G$ be a group satisfying condition ($*$).  Further assume that some compact open subgroup of $G$ is topologically finitely generated.  Then $\rL(G)$ is a factor and $\rS(\rL(G)) = \ol{\Delta(G)}$.
  \begin{itemize}
  \item If $G$ is discrete, then $\rL(G)$ is a type ${\rm II}_1$ factor.
  \item If $G$ is unimodular but not discrete, then $\rL(G)$ is a type ${\rm II}_\infty$ factor.
  \item If $\Delta(G) = \lambda^\ZZ$ for some $\lambda \in (0,1)$, then $\rL(G)$ is a type ${\rm III}_\lambda$ factor.
  \item If $\Delta(G)$ is not singly generated, then $\rL(G)$ is a type ${\rm III}_1$ factor.
  \end{itemize}
\end{theorem}
\begin{proof}
  Let $\vphi$ be a Plancherel weight on $\rL(G)$.  We first show that $\rL(G)$ is a factor.  Assume that this is not the case.  Then there is a central projection $z \in \rL(G) \setminus \{0,1\}$.  Since $z$ is central, $\psi := \vphi(z \cdot)$ is a weight with the same modular automorphism group as $\vphi$.  Hence $\psi$ restricts to a KMS-weight on $\Cstarred(G)$.  By Theorem \ref{thm:unique-KMS-state} there is a scalar $c \in \RR_{> 0}$ such that $\psi|_{\Cstarred(G)}  = c \cdot \vphi|_{\Cstarred(G)}$.   Since $\vphi$ is faithful, we have $\psi(1-z) \neq 0$, which contradicts the definition of $\psi$.  We have shown that $\rL(G)$ is a factor.

  We next show that $\rL(G_0)$ is a factor.  Let $K \leq G$ be a compact open subgroup such that $\cN_G(K)/K$ contains a element of infinite order.  Since $K$ is compact and open, $\cN_G(K) \leq G_0$.  So $G_0$ satisfies condition (*) and the first part of the proof implies that $\rL(G_0)$ is a factor.  We may hence apply Theorem \ref{thm:S-invariant-group-von-Neumann-algebra}, implying that $\rS(\rL(G)) = \ol{\Delta(G)}$.

  According to Section \ref{sec:type}, it only remains to determine the type of $\rL(G)$ in case $G$ is unimodular.  If $G$ is discrete, then $\rL(G)$ is a type ${\rm II}_1$ factor.  If $G$ is unimodular but not discrete, then $\rL(G)$ is a factor with a faithful properly infinite trace.  Hence $\rL(G)$ is of type ${\rm I}_\infty$ or type ${\rm II}_\infty$.  Let $K \leq G$ be a compact open subgroup.  We show that the finite factor $p_K \rL(G) p_K$ is not of type ${\rm I}$.  Then it is of type ${\rm II}_1$ and hence $\rL(G)$ is of type ${\rm II}_\infty$.  By Propositions \ref{prop:characterisation-fixed-point-with-open-stabiliser} and \ref{prop:characterisation-both-fixed-points-with-open-stabiliser}  there is a hyperbolic element $g \in G$ such that $[K:K \cap g^n K g^{-n}]$ is bounded for  $n \in \ZZ$.  Since $g$ is hyperbolic, $g^n \ra \infty$ in $G$.  So for all $\xi, \eta \in \Ltwo(G)$ we have
  \begin{equation*}
    \langle p_K u_g^n p_K \xi, \eta \rangle
    =
    \langle u_g^n p_K \xi, p_K \eta \rangle 
    \ra
    0
    \eqstop
  \end{equation*}
  So $p_K u_g^n p_K \ra 0$ weakly.  At the same time, Lemma \ref{lem:invertible-average} says that $p_K u_g^{-n} p_K u_g^n p_K \geq {[K : K \cap g^n K g^{-n}]^{-2} p_K}$ is bounded from below and cannot converge to $0$.  Put differently, the sequence $(p_K u_g^n p_K)_n$ does not converge to $0$ in the strong topology.  Since on finite type ${\rm I}$ factors the strong and the weak topology coincide, we conclude that $p_K \rL(G) p_K$ must be of type ${\rm II}_1$.
\end{proof}

\section{Non-amenability}
\label{sec:non-amenability}

The following theorem gives a non-amenability criterion for a group von Neumann algebra of a locally compact group.  It is based on the fact that $\rL(G)$ is amenable if and only if $G$ is amenable, as long as $G$ is supposed to be discrete.
\begin{proposition}
  \label{prop:non-amenability-general}
  Let $G$ be a locally compact group containing some compact open subgroup $K \leq G$ such that $\cN_G(K)$ is not amenable.  Then $\rL(G)$ is not amenable.
\end{proposition}
\begin{proof}
   Take $K \leq G$ as in the statement of the proposition.  Put $H := \cN_G(K)$.  Since $H \leq G$ is open, there is a normal conditional expectation $\rL(G) \ra \rL(H)$.  So it suffices to prove that $\rL(H)$ is non-amenable.  Since $K \unlhd H$, we have $p_K \rL(H) p_K \cong \rL(H/K)$, which is a non-amenable von Neumann algebra.  It follows that also $\rL(H)$ is non-amenable, which finishes the proof.
\end{proof}

We can apply our non-amenability criterion to groups acting on trees.
\begin{theorem}
  \label{thm:non-amenability}
  Let $G$ be a group satisfying condition ($*$) and assume that some compact open subgroup of $G$ is topologically finitely generated.  Then $\rL(G)$ is not amenable.
\end{theorem}
\begin{proof}
  We show that there is a compact open subgroup $L \leq G$ and transverse hyperbolic elements $g, h \in \cN_G(L)$.  Fix $K \leq G$ a compact open subgroup.  By Lemma \ref{lem:many-hyperbolic-elements} there are pairwise transverse hyperbolic elements $(g_i)_{i \geq 1}$ and $r \in \NN$ such that $[K : K \cap g_i^l K g_i^{-l}] \leq r$ for all $i \geq 1$ and $l \in \ZZ$.  For fixed $i$ the subgroup $K_i := \bigcap_{l \in \ZZ} g_i^l K g_i^{-l} \leq K$ is normalised by $g_i$.  Proposition \ref{prop:characterisation-top-finitely-generated} implies that $K$ is topologically finitely generated.  So by Lemma~\ref{lem:control-number-subgroups}, there are only finitely many closed subgroups of index less or equal to $r$ in $K$.  Hence $K_i$ is an intersection of finitely many open subgroups of $K$.  It is hence open in $K$ and $[K:K_i]$ is bounded by a constant only depending on $r$ and $K$.  We hence find different indices $i,j \geq 1$ such that $K_i = K_j$.  Put $L := K_i$, $g := g_i$ and $h := g_j$.

  Since $g,h$ are transverse hyperbolic elements the ping-pong lemma implies that $\langle g^n,h^n \rangle \cong \freegrp{2}$ for some $n \in \NN$.  Moreover, we may assume that each element of $\langle g^n, h^n \rangle$ is hyperbolic.  This implies that $\langle g^n, h^n \rangle \cap K = \{e\}$.  So $H := \langle g, h , L \rangle$ is an open non-amenable subgroup of $G$ and $K \unlhd H$ is normal.  We can hence apply Proposition \ref{prop:non-amenability-general}.  This finishes the proof. 
\end{proof}

\section{Schlichting completions of Baumslag-Solitar groups}
\label{sec:BS-groups}

In this section we show that Schlichting completions $\rG(m,n)$ of non-amenable Baumslag-Solitar groups satisfy condition ($*$).  They are hence the first examples of non-discrete \Cstar-simple groups.  We further calculate the type of the factors $\rL(\rG(m,n))$.  In unpublished work with C.Ciobotaru, we obtained factoriality of $\rL(\rG(m,n))$ and could calculate its type by different methods.

It is possible to give a criterion for graphs of groups with finite index inclusions of edge groups into vertex group, so as to make sure that the Schlichting completion of its fundamental group satisfies condition ($*$).  This method gives rise to further examples to which our main result Theorem \ref{thm:cstar-simplicity} applies.

Let $2 \leq |m| \leq n$ be natural numbers.  Then the Baumslag-Solitar group
\begin{equation*}
  \mathrm{BS}(m,n) := \langle a, t \amid t a^m t^{-1} = a^n \rangle
\end{equation*}
contains the commensurated subgroup $\langle a \rangle$.  Denote by
\begin{equation*}
  \rG(m,n) = \mathrm{BS}(m,n) \sslash \langle a \rangle \geq \rK(m,n) = \langle a \rangle
\end{equation*}
the Schlichting completion  of the pair $\mathrm{BS}(m,n) \geq \langle a \rangle$.  Recall that $\rG(m,n) \leq \mathrm{Sym}(\mathrm{BS}(m,n)/\langle a \rangle)$ is a closed subgroup.  By Bass-Serre theory, $\mathrm{BS}(m,n)/ \langle a \rangle$ has the structure of a $m+ n$-regular tree $T$, were $g t \langle a \rangle$ and $g \langle a \rangle$ are connected by an edge for all $g \in \mathrm{BS}(m,n)$.  It follows that $\rG(m,n) \leq \Aut(T) \leq \mathrm{Sym}(\mathrm{BS}(m,n)/\langle a \rangle)$.

The next lemma describes the image of the modular function of $\rG(m,n)$.
\begin{lemma}
  \label{lem:image-modular-homomorphism}
  Let $m,n \in \ZZ^\times$.  Then the image of the modular function of $\rG(m,n)$ is $\left | \frac{m}{n} \right |^\ZZ$.
\end{lemma}
\begin{proof}
  If $|m| = |n|$, then $\rG(m,n)$ is discrete and hence it is unimodular.  We may hence assume that $1 \leq |m| < n$.  Then $\mathrm{BS}(m,n) \subset \rG(m,n)$ is a dense subgroup.  Since $\left ( \frac{m}{n} \right )^\ZZ \subset \RR_{> 0}$ is discrete it suffices to show that $\Delta(\mathrm{BS}(m,n)) = \left ( \frac{m}{n} \right )^\ZZ$.  Since $a \in \cN_{\rG(m,n)}(\rK(m,n))$, it follows that $\Delta(a) = 1$.  Further $t^{-1} \ol{\langle a^n \rangle} t = \ol{\langle a^m \rangle}$ showing that $\Delta(t) = \left | \frac{m}{n} \right |$.  Since $\mathrm{BS}(m,n)$ is generated by $a,t$, this finishes the proof of the lemma.
\end{proof}

\begin{theorem}
  \label{thm:BS-groups}
  Let $2 \leq |m| \leq n$ and consider the relative profinite completion $\rG(m,n)$ of the Baumslag-Solitar group $\mathrm{BS}(m,n)$.  Then the following statements are true.
  \begin{itemize}
  \item $\rL(\rG(m,n))$ is a non-amenable factor.
  \item If $|m| = n$, then $\rG(m,n)$ is discrete and $\rL(G(m,n))$ is of type ${\rm II}_1$.
  \item If $|m| \neq n$, then $\rL(G(m,n))$ is of type ${\rm III}_{\left | \frac{m}{n} \right |}$.
  \item $\Cstarred(\rG(m,n))$ is simple.
  \end{itemize}

\end{theorem}
\begin{proof}
  First note that $\rG(m,n) \cong \ZZ/n\ZZ * \ZZ$ if $|m| = n$.  So in this case $\rG(m,n)$ is icc and non-amenable, showing that $\rL(\rG(m,n))$ is a non-amenable type ${\rm II}_1$ factor.  Moreover, $\rG(m,n)$ is a Powers group in the sense of de la Harpe \cite{delaharpe85}.  So $\Cstarred(\rG(m,n))$ is simple.

  In case $|m| \neq n$ we have $\mathrm{BS}(m,n) \leq \rG(m,n)$.  Note that $\rG(m,n)$ acts transitively on $T$, so that $\rG(m,n) \grpaction{} \partial T$ is minimal by Proposition \ref{prop:equivalence-minimal-action}.  The element $atat^{-1} \in \rG(m,n)_0$ is hyperbolic and normalises the closure of $\langle a^n \rangle$, which is open in $\rK(m,n)$.  So Proposition \ref{prop:characterisation-fixed-point-with-open-stabiliser} shows that the fixed points of $atat^{-1}$ have an open stabiliser in $G$.  This verifies condition ($*$) for $\rG(m,n)$.  Since $\rK(m,n)$ is topologically singly generated, Theorems \ref{thm:factoriality}, \ref{thm:cstar-simplicity} and \ref{thm:non-amenability} apply.

  Since $\Delta(\rG(m,n)) = \frac{|m|}{n}^\ZZ$ by Lemma \ref{lem:image-modular-homomorphism}, Theorem \ref{thm:factoriality} shows that $\rL(\rG(m,n))$ is a type ${\rm III}_\lambda$  factor for $\lambda = \frac{|m|}{n}$.  By Theorem \ref{thm:non-amenability}, $\rL(\rG(m,n))$ is not amenable.  Theorem \ref{thm:cstar-simplicity} implies that $\Cstarred(\rG(m,n))$ is simple.

\end{proof}

\bibliographystyle{mybibtexstyle}
\bibliography{operatoralgebras}

\vspace{2em}
{\small \parbox[t]{200pt}
  {
    Sven Raum \\
    Westf{\"a}lische Wilhelmsuniversit{\"a}t M{\"u}nster \\
    Fakult{\"a}t Mathematik und Informatik \\
    Einsteinstra{\ss}e 62 \\
    D-48149 M{\"u}nster \\
    Germany \\
    {\footnotesize sven.raum@gmail.com}
  }
}

\end{document}